\documentclass[numreferences]{kluwer}
\usepackage{graphicx}
\usepackage{setspace}
\usepackage{enumitem}
\usepackage{amsthm}
\newtheorem{Thm}{Theorem}[section]

\newtheorem{Lem}[Thm]{Lemma}

\newtheorem{Alg}[Thm]{Algorithm}

\newtheorem{Rem}{Remark}

\begin{document}
	%\begin{article}
		\begin{opening}
			\title{On Trigonometric Interpolation and Its Applications}
			\author{Xiaorong \surname{Zou}\email{xiaorong.zou@bofa.com}}
			%\author{Xiaorong \surname{Zou}\email{xiaorzou@gmail.com}}
			\institute{Global Market Risk Analytic, Bank of America}
			\runningauthor{X. Zou}
			\runningtitle{TIBO for Non-Linear ODE System}
			\date{May 4, 2025}
			\classification{MSC2000}{Primary 65T40; Secondary 65L05}
			\keywords {Fourier Series,	Trigonometric Interpolation, Fast Fourier Transformation (FFT),  Ordinary Differential Equation,  Runge-Kutta method.}
\begin{abstract}
In this paper, we propose a new trigonometric interpolation algorithm and establish relevant convergent properties.  The method adjusts an existing trigonometric interpolation algorithm such that it can better leverage Fast Fourier Transform (FFT)  to enhance efficiency. The algorithm can be formulated in a way such that certain cancellation effects can be effectively leveraged for error analysis, which enables us not only to obtain the desired uniform convergent rate of the approximation to a function, but desired uniform convergent rates for its derivatives as well. 

We further enhance the algorithm so it can be applied to non-periodic functions defined on bounded intervals. Numerical testing results confirm decent accurate performance of the algorithm.  For its application, we demonstrate how it can be applied to estimate integrals and solve linear/non-linear ordinary differential equation (ODE). The test results show that it significantly outperforms Trapezoid/Simpson method on integral and standard Runge-Kutta algorithm on ODE. In addition, we show some numerical evidences that estimation error of the algorithm likely exhibits ``local property", i.e. error at a point tends not to propagate, which avoids significant compounding error at some other place, as a remarkable advantage compared to polynomial-based approximations.
\end{abstract}
\end{opening}

\iffalse
\begin{highlights}
\item Introduce a new trigonometric interpolation algorithm that can be carried out by Fast Fourier Transform for optional performance;  establish relevant properties, especially showing that it has desired convergence rate for both interpolant and high order derivatives of interpolant, which provides theoretic support for its applications. (Part I);
\item Extend the algorithm so it can be used to non-periodic function (Part II);
\item Study numerical  performance of the algorithm (Part II);
\item Study applications of the algorithm, including to estimate integrals and solve linear/non-linear ordinary differential equation (ODE). Test results show that it outperforms Trapezoid/Simpson method to copy with integral and standard Runge-Kutta algorithm in handling ODE. (Part II)
\end{highlights}
\fi
\section{Introduction}\label{sec:introduction}
Let $f(x)$ be a  periodic function with period $2b$. Assume that its $K+1$-th derivative $f^{(K+1)}$ is bounded by $D_{K+1}$ for certain $K\ge 1$.  By classic Fourier analysis, it can be represented by the Fourier series:
\begin{equation}\label{fx}
	f(x)  = \frac{A_0}2 + {\sum_{j \ge 1}} A_j \cos\frac{j\pi x}b + B_j \sin\frac{j\pi x}b,
\end{equation}
where
\begin{eqnarray*}\label{Aj1}
	A_j&=&\frac{1}{b}\int^{b}_{-b} f(x)\cos\frac{j\pi x}bdx \quad j\ge 0 \label{coef_cos_A}, \\
	B_j&=&\frac{1}{b}\int^{b}_{-b} f(x)\sin\frac{j\pi x}bdx \quad j>0.  \nonumber
\end{eqnarray*}
It is well-known that $A_j, B_j$ converge to $0$ with the order $O(\frac{1}{j^{K+1}})$. One can estimate $f(x)$ by the sum of first $2n+1$ terms
%\begin{equation*} \label{tildef}
\[
	f(x)  \approx \tilde {f}_n(x) := \frac{A_0}2 + {\sum_{1\le j \le n}} A_j \cos\frac{j\pi x}b + B_j \sin\frac{j\pi x}b.
%\end{equation*}
\]
A function like $\tilde {f}_n(x)$ is called a trigonometric polynomial of degree $n$ in literature \cite{Anthony}. 
In practice, the coefficients $A_j, B_j$ need to be estimated by some numerical algorithm. Let 
%\begin{equation*}\label{grid_2n}
\[
	x_j =\frac{2bj }{2n+1} = j\lambda, \qquad  \lambda:=\frac{2b}{2n+1}, \qquad  -n\le j \le n.  
\]
%\end{equation*}

be $2n+1$ equispaced nodes over $[-b,b]$, and $y_j =f(x_j)$. One can estimate  
\begin{eqnarray}
	A_j &\approx& {a}^{(2n+1)}_{j}:= \frac{2}{2n+1}\sum_{k=-n}^{n}y_k\cos\frac{2\pi kj}{2n+1}, \label{formula_tra_a}\\
	B_j &\approx& {b}^{(2n+1)}_{j}: = \frac{2}{2n+1}\sum_{k=-n}^{n}y_k\sin\frac{2\pi kj}{2n+1}. \label{formula_tra_b}
\end{eqnarray}
It is not hard to verify that
%\begin{equation*}\label{hatf}
\[	
	f^{(2n+1)}_n := \frac{{a}^{(2n+1)}_{0}}2 + {\sum_{1\le j \le n}} {a}^{(2n+1)}_{j} \cos\frac{j\pi x}b + {b}^{(2n+1)}_{j} \sin\frac{j\pi x}b.
\]
%\end{equation*}
is the unique trigonometric interpolant of degree $n$ that solves 
\[
f^{(2n+1)}_n(x_j) =f(x_j).
\]
In general, a trigonometric interpolant of degree $n$ with given nodes  $t_0<t_1\dots, <t_{N-1}$ is a trigonometric polynomial  
%\begin{equation*}\label{2nplus1}
\[
f_n(x)=a_{0,n} + {\sum_{1\le j \le n}}  a_{j,n} \cos\frac{j\pi x}b + b_{j,n} \sin\frac{j\pi x}b,
\]
%\end{equation*}
such that 
%\begin{equation*}\label{ykff}
\[
y_k:=f(t_k)=f_n(t_k), \quad k=0,\cdots, N-1.
\]
%\end{equation*}
The number $N$ of the nodes is often required to be same as number of the coefficients of $f_n(x)$ to ensure the uniqueness, which is often desired. For trigonometric interpolation,  equispaced nodes are recommended \cite{WJMT}: 
%\begin{equation*}\label{grid_t}
\[
	t_j = \alpha + \frac{2bj }{N},  \qquad  0\le j < N,
\]
%\end{equation*}
where $\alpha$ is a fixed constant.  Performance of an interpolation is quite sensitive to the selection of $\alpha$.  For example,  it is shown in \cite{Gosselin} that there is a continuous periodic function $f(x)$ for which the unique interpolating sequence with $\alpha=0$ diverges for all $x\neq 0  mod(2\pi)$, but the sequence converges uniformly if $\alpha$ is irrational with respect to $\pi$.

$\{{a}^{(2n+1)}_{j}, {b}^{(2n+1)}_{j}\}$ by Eq. (\ref{formula_tra_a}-\ref{formula_tra_b}) can be easily implemented by Fast Fourier Transform (FFT). Note that odd number $N=2n+1$ of terms occurs in FFT while the optimal performance of FFT can only be reached when $N=2^h$ is a radix-$2$ integer for some $h$ \cite{ct}.  A trigonometric interpolation by even nodes is available by Theorem 3.5-3.6 \cite{ryts} combined as follows :
\begin{Thm}\label{ryts_th_3_6} Let $f(x)$ be a periodic function with period $2b$ and $N=2M$ be an even integer and define	
	%\begin{equation*}\label{ryts_x_grid_N}
\[
	x_j = -b + j\lambda,  \quad \lambda= \frac{2b}N,  \quad  y_j = f(x_j), \qquad 0\le j <N.
\]
	%\end{equation*}
	Then there is a unique  trigonometric polynomial defined by
	\begin{eqnarray}
		\tilde{Q}_M(x) &=& \sum_{0\le j \le M}a^e_j \cos\frac{j\pi x}b + \sum_{1\le j <M}a^o_j \sin\frac{j\pi x}b, \label{ryts_f_M_interpolation_new}\\
		a^e_0&=& \frac1N\sum_{0\le j <N}y_{j}, \label{ryts_a0}\\
		a^e_j&=&\frac{2}N\sum_{0\le k <N}(-1)^j y_k \cos\frac{2\pi j k}{N},  \quad 1\le j <M,  \label{ryts_aj_even}\\
		a^e_M &=& \frac1N\sum_{0\le j <N}(-1)^jy_{j}, \label{ryts_aM}\\
		a^o_j&=&\frac{2}N\sum_{0\le k <N} (-1)^j y_k \sin\frac{2\pi j k}{N}, \quad 1\le j <M, \label{ryts_b}
	\end{eqnarray}
	such that 
	\[
	\tilde{Q}_M(x_{k})  =y_{k}, \qquad 0\le k < N.
	\]
	Furthermore, the error is bounded by
	\begin{equation}\label{ryts_error}
		|R_N(x)|:=|f(x)-\tilde{Q}_n(x)|\le \frac{\xi_M}{M^{K-\frac12}}, \qquad \xi_M=o(1). 
	\end{equation}
\end{Thm} 
Hence, the error $R_N(x)$ converges to zero uniformly with respect to x as $N\to \infty$ if $K\ge 1$, and the rate of convergence of $\tilde{Q}_n(x)$ automatically becomes faster for smoother $f$, remarkable advantage compared to polynomial interpolation when the convergence rate is limited by the degree of the polynomial \cite{ryts}.

Theorem  \ref{ryts_th_3_6} is only suitable for the reconstruction of periodic functions. For a function defined over bounded interval, which can always be scaled to $[-1,1]$,  one can first equivalently replace $f(x)$ by some smooth periodic function. A well-known method to transform $f(x)$ to a periodic function $F(\phi)$ is as follows:
\begin{equation}\label{cos_transform}
	F(\phi) = f(\cos \phi),\quad x=\cos\phi, \qquad \phi \in [0, \pi].
\end{equation}
It is easy to see that $F(\phi)$ can be interpreted as an even $2\pi$-periodic function of $\phi$ with same smoothness as $f(x)$. Theorem  \ref{ryts_th_3_6} can be applied to $F(\phi)$ with 
\[
\phi_j = -\pi + j\lambda,  \quad \lambda= \frac{2\pi}N,  \quad  y_j = F(\phi_j), \qquad 0\le j <N.
\]
Since  $F(\phi)$ is even,  Eq. (\ref{ryts_f_M_interpolation_new}) is reduced to
\[
\tilde{Q}_n(\phi) = \sum_{0\le j \le M}a^e_j \cos (j \phi),
\]
where $a_j$ is defined by Eq. (\ref{ryts_a0})-(\ref{ryts_aM}).  By introducing Chebyshev polynomials, one obtains a polynomial interpolation 
%\begin{equation*}\label{Chebyshev_interp}
\[
	P_n(x) = \sum_{0\le j \le M}a^e_j T_j(x),
\]
%\end{equation*}
where
\[
T_j(x) = \cos (j\phi) =\cos (j \arccos x), \qquad k=0,1,2, \dots, 
\]
are $j-th$ polynomial that were first introduced by Chebyshev: 
\[
T_0(x)=1, \quad T_1(x) = x, \quad T_2(x) = 2x^2-1, \quad T_3(x) = 4x^3-3x,\quad, \dots.
\] 
The error estimation (\ref{ryts_error}) can be directly applied to the Chebyshev polynomial interpolation for same convergence rate as in trigonometric interpolation.  In fact, the error rate can be slightly improved  as shown in \cite{ryts}.

Trigonometric interpolation is believed to be suitable for periodic function while Chebyshev polynomial interpolation tends to be preferred for non-periodic function in the literature. Some comparison between the two interpolation can be found in \cite{WJMT}. We highlight some pros and cons of trigonometric interpolation (T) and Chebyshev polynomial interpolation (C).
\begin{enumerate}
	\item  T can be easily implemented using FFT while C need copy with Chebyshev polynomials without close-form expression.  
	\item Attractive analytic representation greatly benefits T in its applications, especially where the analytic close form of derivatives and integrals of target function is desired.  On the contrary, implicit representation of Chebyshev polynomial restricts its applications in large degree.
	\item \label{limitation_item3} C can be applied to general functions, but T is limited to periodic functions.  
	\item The performance of both C and T depends on the smoothness of target function on the whole domain.  Lack of the smoothness at a single point would deteriorate overall performance dramatically. 
\end{enumerate}
Trigonometric interpolation with even number of points are discussed through imposing certain constrain to ensure uniqueness in \cite{Anthony}. The interpolant is constructed by trigonometric Lagrange basis function and therefore is more flexible with non-equspaced points. Recent researches on trigonometric interpolation are more based on the pioneer work \cite{berr11}. For a complex-valued function on some interval $I$,  the interpolants are constructed in a barycentric form for a give set of interpolation nodes.  A trigonometric interpolation algorithm is introduced in \cite{badd13} based on constructing a barycentric rational approximant selecting the nodes progressively via a greedy algorithm. 
In \cite{balt18}, a set of interpolating points is introduced to construct linear rational trigonometric interpolant written in barycentric form with exponential convergence rate and has been used effectively in [19] to interpolate functions on two-dimensional starlike domains. 

In addition, some recent researches have been done on Hermite interpolation.  Note that Hermite interpolation is an interpolates not only a function on a given set of nodes, but the values of its first $m$ derivatives with certain integer $m\ge 0$. \cite{elef} provides an algorithm to construct a barycentric trigonometric Hermite interpolant via an iterative approach.  More references in that direction can be found in \cite{scwe1}-\cite{dedinosi4}.

Despite the advantage on convergence rate and flexibility of usage, non-classic interpolants such as with barycentric form is less computationally attractive compared to the classic trigonometric interpolation described in Theorem \ref{ryts_th_3_6}, which can be easily implemented by FFT.  The major disadvantage of classic trigonometric interpolation is Limitation \ref{limitation_item3} mentioned above. 

As the first part of our study on trigonometric interpolation and its applications,  we introduce an adjusted version of Theorem \ref{ryts_th_3_6} that can be carried out by FFT with optimal operations. The adjustment generates some minor error on half of the nodes as downside. But advantage is quite significant. It aligns the number of interpolation nodes with degree of trigonometric polynomial so that we can leverage FFT power more efficiently. For performance analysis, the alignment enables us to leverage periodicity and symmetry of certain quantities occurred in study of estimation error. For example, we are able to derive a simple relation of interpolants when grid points are doubled (Lemma \ref{lemma-key-pattern}) and improve the error rate from $M^{-K+\frac12}$ in Theorem \ref{ryts_th_3_6} to $M^{-K}$. More importantly, we are able to establish uniform convergence theorem with rate $O(\frac{1}{N^{K-k}})$ for $k$-th ($k<K$) derivative of $f$ as well, and thus provide a theoretic support for certain applications such as solving differential/integral equations or searching a optimal solution where the gradient of target function need to be carried out efficiently. 

In the second part of this study, we shall enhance the algorithm to cover a non-periodic function $f$ defined over a bounded interval $[s,e]$. Instead of transforming $f$ to a periodic function described in Eq (\ref{cos_transform}),  we assume that $f$ can be smoothly defined over $[s-\delta, e+\delta]$ for some $\delta>0$. We then extend periodically $f$ by a cut-off function to keep the smoothness of $f$.  Compared to Chebyshev polynomials, the algorithm outputs a trigonometric interpolant with high rate convergence as well as a attractive analytic form. Note that the algorithm can be applied to piece-wise smooth functions by restricting to each of smooth parts separately. Such extension is not unique and we propose a simple solution with quite decent testing results as shown in Section \ref{sec:Numerical_Performance}.

Considering the analytic attractiveness of trigonometric polynomial,  especially in handling differential and integral operations, we expect that the proposed trigonometric estimation of a general function can be used in a wide spectrum.  As a starting point for subsequential applications, we show that it can be used to estimate a general integral with stabler and more accurate performance than popular Trapezoid and Simpson rules in copying with volatile integrands.  A new algorithm (Algorithm \ref{alg_ode}) is proposed to solve first order linear ordinary differential equation (ODE) since the solution can be formulated by integral.  For a more sophisticated application, we make effort to develop an algorithm (Algorithm \ref{alg:ode_general}) to solve a general first order non-linear ODE by searching a trigonometric representation of target function. The algorithm is formulated in a way to fully leverage FFT so it can be carried out efficiently.  The solution is expected to converge more quickly than the standard Runge-Kutta method (RK4) \cite{ptvf}. The testing results confirm that it outperforms significantly not only RK4, but a modified RK4 where estimation at each step is based on the true value of previous step (See details in Section \ref{application_ode_general}). The same optimization method of solving first order non-linear ODE can be extended to solve first order non-linear $d$-dim ODE system \cite{zou_ODE_system}.  Similar idea can be applied to approximate solutions of a high order non-linear ODE with flexible initial and boundary conditions \cite{zou_ODE}. The trigonometric interpolation can also be used to transform high order linear integro-differential equations to linear algebraic system and thus can be used to solve Fredholm/Volterra type IDE effectively \cite{zou_ODE_ide_fredholm} and \cite{zou_ODE_ide_volterra}.  

For a general non-linear ODE over a bounded interval $[s,e]$ with certain combination of initial and boundary conditions, one can first extend the variable domain and then search through similar optimization process for a periodic solution such that its restriction to $[s,e]$ solves original ODE. The optimization error function is derived from the ODE and captures desired structure among the target function and its relevant derivatives. As such,  the method is not limited to linear ODE as long as the gradient vector can be efficiently carried out, which, fortunately, can be done through FFT as appears in first order non-linear ODE.  It is worthwhile to compare the optimization method to a popular reproducing kernel method in solving high order ODE with boundary value conditions.  In a nutshell, the latter applies certain reproducing kernel to construct orthogonal bases for a target Hilbert space and the solution of certain ODE problem can be estimated by the projection to a finite subspace generated by certain number of basis functions.  The solution can be obtained by solving a linear system based on linearity of the associated differential operator. As such,  two methods are similar in the sense that they both estimate the solution through basis function in a Hilbert space,  but kernel based method is often limited to linear ODE with boundary value conditions while the optimization method can be applied to any high non-linear ODE with combination of initial and boundary conditions in theory. See  \cite{ode_yuyuse}, \cite{ode_desa}, \cite{ode_zhlish}, \cite{ode_desa}, \cite{ode_bumozh13} and \cite{ode_liwu36} for some reproducing kernel method for solving higher order linear ODE and Integro-differential equations

An optimal error estimation of trigonometric interpolations can be quite challenging. Certain cancellation effect, which avoids error propagation, might be one of the major reasons why the performance tends to be better than what can be estimated vigorously, as explained in Item \ref{rem:test_2_performance_part5} of Remark \ref{rem:test_2_performance}. The challenge to sufficiently leverage the cancellation effect could be the major obstacle for an optimal performance analysis.  The error estimation in this study depends more on skillfully maneuvering saw-tooth qualities and simplifying them in proper format.  Without loss of generosity, we assume the target periodic function $f$ is either even or odd since any function can always be decomposed as the sum of an even and odd functions. The symmetry and periodicity of $f$ play important rules in the analysis conducted in this study. 

The study is structured into two major parts. The first part focuses on the establishment of main theorems and is organized as follows.  In Section \ref{sec:main_results}, we present the major results, including the new trigonometric interpolation algorithm on periodic functions mentioned above, relevant convergence properties, as well as the enhancement of the algorithm for its use to non-periodic functions and its performance.  Section \ref{sec:The_Development_of_the_Main_Theorems} is mainly used to establish the algorithm (Theorem \ref{main_thm}),  estimations on coefficients of interpolants (Theorem \ref{converence_coef}) and convergence properties (Theorem \ref{conv_k_new}). The proof depends on a few key equations whose derivations are partly moved to \ref{pf_keylemma}. 

The second part of this study is organized as follows. Section \ref{sec:Trigonometric_Interpolation_of_General_Functions} is used to develop Algorithm \ref{alg_textenion} for the enhancement of trigonometric interpolation for non-periodic function.  In Section \ref{sec:Numerical_Performance}, we conduct some numerical  tests. The result confirms that the performance of Algorithm \ref{alg_textenion} is sensitive to the smoothness of $f$ and is quite satisfactory when $f$ is sufficient smooth. In addition, we explain that error of trigonometric approximation often exhibits cancellation effect and thus does not propagate and generate significant compounding errors, a remarkable advantage compared to polynomial-based approximation.  Section \ref{sec:Applications} is devoted to the applications outlined above on the third purpose of this paper.  The summary is made on Section \ref{sec:Conclusions}.

\section{Main Results}\label{sec:main_results}
\label{sec:main}
%\subsection{Trigonometric Interpolation on periodic functions}\label{subsec:main-period-function}
In this subsection,  unless otherwise specified, $f(x)$ denotes either an even or odd periodic function with period $2b>0$ and its $K+1$-th derivative $f^{(K+1)}(x)$ exists and is bounded by $D_{K+1}$ for some positive integer $K\ge 1$.  Eq. (\ref{fx}) is reduced to 
\begin{eqnarray*}\label{fxeven}
	f(x)=\left\{\begin{array}{cc}
		\frac{A_0}2 + \sum_{ j \ge  1} A_j \cos\frac{\pi j x}b, & \mbox{if $f$ is even,}  \\
		\sum_{j \ge 1} B_j \sin\frac{\pi j x}b, & \mbox{if $f$ is odd.}
	\end{array}\right.
\end{eqnarray*}

For a given even integer $N=2M=2^{q+1}$, Theorem \ref{ryts_th_3_6} provides us trigonometric interpolant that recovers the following grid nodes:  
\begin{eqnarray}
	x_j &:=& -b + j\lambda,  \quad \lambda= \frac{2b}N , \quad 0\le j <N, \label{x_grid_N} \\
	y_j &:=& f(x_j),  \label{f_N_interpolation_new} \\
	y_j&=& \left\{\begin{array}{cc} \nonumber%\label{symmetry}
		y_{N-j} & \mbox{if $f$ is even,}  \\
		-y_{N-j} & \mbox{if $f$ is odd.}
	\end{array}\right.
\end{eqnarray}

If $f$ is odd, then $f(-b)=f(0)=0$ and there are $M-1$ free points $\{(x_j,y_j)\}_{1\le j<M}$, aligned with the number of coefficients in Eq. (\ref{ryts_b}) and $a^o_j$ can be solved by FFT with optional operations. 

The situation for even case is slightly different.  There are $M+1$ free points $\{(x_j,y_j)\}_{0\le j\le M}$ aligned with $M+1$ coefficients to ensure uniqueness of interpolant. There are
two undesired features for us to conduct error analysis. 1) $a^e_M$ defined by Eq. (\ref{ryts_aM}) is not consistent with derived $a^e_M$ by Eq. (\ref{ryts_aj_even}), and 2) there are odd number $M+1$ terms in $\tilde{Q}_M(x)$ defined by Eq. (\ref{ryts_f_M_interpolation_new}).  A solution to address this issue is to combine first and last terms of $\tilde{Q}_M(x)$ whose impact on  node $x_k$ is
\[
a_0+a_M\cos\frac{\pi Mx_k}b  =a_0 + (-1)^ka_M = \left\{\begin{array}{cc}
	\frac{1}M \sum_{0\le j < M}y_{2j}& \mbox{if $f$ is even,}  \\
	\frac{1}M \sum_{0\le j < M}y_{2j+1} & \mbox{if $f$ is odd.}
\end{array}\right.
\]
By replacing $a^e_0+a^e_M\cos\frac{\pi M x}{b}$ by $a_0=\frac{1}M \sum_{0\le j < M}y_{2j}$, and keeping other coefficients, we obtain a new polynomial that fits to all even nodes $x_{2j}$ and approximate to all odd nodes $x_{2j+1}$ with a uniform error:
\begin{equation}\label{eps_M}
	\epsilon_M = \frac{1}M \sum_{0\le j < N} (-1)^j y_{j} 
\end{equation}
We thus obtain  
\begin{Thm}\label{main_thm} Let $f(x)$ be a periodic function with period $2b$ and $N=2M$ be an even integer and $x_j, y_j$ are defined by Eq. (\ref{f_N_interpolation_new}).
	\begin{itemize}
		\item   If $f(x)$ is even, then there is a unique $M-1$ degree trigonometric polynomial
		\begin{eqnarray}
			f_M(x) &=& \sum_{0\le j <M}a_j \cos\frac{j\pi x}b, \label{f_M_interpolation_new}\nonumber\\
			a_0&=& \frac1M\sum_{0\le j <M}y_{2j}, \label{a0}\\
			a_j&=&\frac{2}N\sum_{0\le k <N}(-1)^j y_k \cos\frac{2\pi j k}{N},  \quad 1\le j <M,  \label{aj_even}
		\end{eqnarray}
		such that for $0\le k <M$,  
		\begin{eqnarray}
			f_M(x_{2k})  &=& y_{2k} ,  \label{error_even_new} \\
			f_M(x_{2k+1}) &=&  y_{2k+1} +\epsilon_M.  \label{error_odd_new}
		\end{eqnarray}
		In another word,  $f_M(x)$ fits to all even grid points and shifts away in parallel from all odd grid points by $\epsilon_M$. 
		\item If $f(x)$ is odd, then there is a unique $M-1$ degree trigonometric polynomial
		\begin{eqnarray*}
			f_M(x) &=& \sum_{0\le j <M}a_j \sin\frac{j\pi x}b, \label{f_M_interpolation_new_odd}\\
			a_j&=&\frac{2}N\sum_{0\le k <N} (-1)^j y_k \sin\frac{2\pi j k}{N}, \quad 0\le j <M \label{aj_odd}
		\end{eqnarray*}
		such that it fits to all grid points, i.e.
		%\begin{equation*}\label{error_even_new_odd} 
		\[	
		f_M(x_{k})=y_{k}, \quad  0\le k <N.
		\]
		%\end{equation*}
	\end{itemize}	
\end{Thm} 
To keep self-contained, we provide an elementary proof of Theorem \ref{main_thm} in Subsection \ref{subsec:proof_main_thm} although it is a direct conclusion of Theorem  \ref{ryts_th_3_6}.  A few remarks are in order.  
\begin{Rem}
	\begin{enumerate}
		\item The algorithm is easy to be implemented and computationally efficient.  We can computer them by Inverse Fast Fourier Transform (ifft),
		\[	
		\{a_j (-1)^j\}_0^{N-1} = \left\{\begin{array}{cc}
			2\times Real (ifft(\{y_k\}_{k=0}^{N-1})), & \mbox{if $f$ is even },  \\
			2\times Imag (ifft(\{y_k\}_{k=0}^{N-1})), & \mbox{if $f$ is odd}.
		\end{array}\right.
		\]
		and replace $a_0$ by Eq (\ref{a0}) if $f$ is even. To fully leverage power of FFT,  $N$ should be a radix-$2$ integer, i.e. $N=2^h$ for a positive integer $h$ and operation cost of ifft is $\frac{N}2\log_2N$ as shown in \cite{ct}. 
		
		\item If $f$ is even,  the error $\epsilon_M$ by Eq (\ref{eps_M}) is $O(\frac1{N^{K+1}})$ by applying following Euler-Maclaurin identity:
		\begin{equation}\label{Euler-Maclaurin}
			h\sum_{0\le l <n-1}f(l h) = \int^{b}_{-b} f(x)dx -(\frac{-2b}{n})^{K+1}\int^{b}_{-b}\tilde{B}_{K+1}(\frac x{2b}) f^{(K+1)}(x)dx
		\end{equation}
		where $n\ge 2$ is positive integer and  $h=2b/n$ and $\tilde B_{K+1}$ is the periodic extension of $K+1$-th Bernoulli polynomial \cite{Rainer_Kress}.
		\item The uniqueness of solution is helpful for certain applications where we depend on some optimization process to find $f_M$, as shown in Section \ref{sec:Applications}.
	\end{enumerate}
\end{Rem}
It is not hard to see $a_j$ ($j\ge 1$) is Trapezoidal approximation of Fourier expansion coefficient $A_j$ or $B_j$.   It is natural to expect that $a_j$ approaches to $0$ (as $A_j$ does) as $j\to \infty$. Theorem \ref{converence_coef} provides a boundary of $a_j$ in $j$ and $N$. 

\begin{Thm}\label{converence_coef}
	Assume that $|f^{(K+1)}(x)|$ exists with an upper bound $D_{K+1}$, then 
	\begin{equation}\label{boundary2}
		|a_j| \le  \frac{C(D_{K+1})}{N^{K+1}\sin^{K+1}\frac{\pi j}{N}}, \quad  1\le j <M,
	\end{equation}
	where $C(D_{K+1})$ is a constant depending on $D_{K+1}$.
\end{Thm}
Notice that $a_j$ depends on $j$ and $N$, and the estimation (\ref{boundary2}) shows how $a_j$ decays to $0$ in two dimensions.  For a given $j$,  one can see $|a_j|$ has order $O(\frac{1}{j^{K+1}})$ as $N\to\infty$, which is consistent to the order of Fourier coefficient $A_j$.  For a given interpolant $f_M$ with a fixed large $N$, magnitude of $a_j$ approaches to $0$ at order $\frac{1}{N^{K+1}}$ as $j\to M$. It is worthwhile to point out the second half coefficients $\{a_j\}_{M/2\le j<M}$ decays uniformly with $\frac{1}{N^{K+1}}$, which is one of key observations to establish convergence Theorem \ref{conv_k_new}.

The proof of Theorem \ref{converence_coef} mainly depends on expressing $a_j$ in term of $K+1$-th forward difference as shown in Eq. (\ref{a_l}) and key ingredient is classic Abel Transform. Details can be found in Section \ref{subsec:proof_converence_coef}.

With Estimation  (\ref{boundary2}), we can show uniform convergence properties of $f_M(x)$ as below. 
\begin{Thm}\label{conv_k_new}
	Assume that $|f^{(K+1)}(x)|$ exists with an upper bound $D_{K+1}$, then 
	\begin{eqnarray}
		|f_M(x) - f(x)| &\le& \frac{C_1({D_{K+1}})}{N^{K}},   \label{detal_M_f} \\
		|f^{(k)}_M(x) - f^{(k)}(x)| &\le& \frac{C_2({D_{K+1}})}{N^{K-k}}, \quad 1\le k <K. \label{der_detal_M_f}
	\end{eqnarray}
	where $C_1({D_{K+1}})$ and $C_2({D_{K+1}})$ are two constants depending on $D_{K+1}$.
\end{Thm}		
The error estimation (\ref{detal_M_f})  is different from \cite{ryts}.  Estimation (\ref{ryts_error}) of Theorem \ref{ryts_f_M_interpolation_new} is based on breakdown $f_M(x)$ into two components,  one is partial sum \footnote{We explain ideas by assuming $f$ is even. }
\[
S(x) = \frac {A_0}2 + \sum_{1\le k \le M} A_k \cos\frac{j\pi x}{b} 
\]
and other is the interpolant of residue $(\delta S)(x):=f(x)-S(x)$.  By uniqueness of interpolant for given set of grid points, one obtains 
\[
f_M(x) = S_M(x) + (\delta S)_M(x) = S(x) + (\delta S)_M(x),
\]
and therefore
\[
|R_M(x)| = |f(x)-f_M(x)| \le |(\delta S)(x)| + |(\delta S)_M(x)|.
\]
The overall convergent rate $N^{-K+0.5}$ of $R_M(x)$ is determined by the rate of $|(\delta S)_M(x)|$ that converges slower than $|(\delta S)(x)|$. Details can be found in \cite{ryts}.

The proof of Theorem \ref{conv_k_new} can be found in Section \ref{subsec:proof_conv_k_new}. We directly copy with $f_M$. As such,  we not only get extra accuracy rate by avoiding  $(\delta S)_M(x)$, but be able to obtain accuracy rate on derivatives.   The adjustment on coefficients on Theorem \ref{ryts_f_M_interpolation_new} makes it possible for us to derive a clean pattern of adjusted coefficients when interpolating grid points are doubled (Lemma \ref{lemma-key-pattern}), which plays a key rule in establishment of Estimation (\ref{detal_M_f})-(\ref{der_detal_M_f}).

%\subsection{Trigonometric Interpolation on non-periodic functions and applications}\label{subsec:main-non-period-function}
For a non-periodic function $f$ over a bounded interval $[s-\delta,e+\delta]$ for some $\delta>0$, we extend $f$ to a periodic function with same smoothness by a cut-off function $h(x)\in C^{\infty}(R)$ with following properties: 
%One can use the following method to periodically extend $g$ by a cut-off function $h$ with the following properties 
\begin{eqnarray*}
	h(x)=\left\{\begin{array}{cc}
		1 & x\in [s,e], \\
		0 & \mbox{$x<s-\delta$ or $x>e+\delta$}. \\
	\end{array}\right.
\end{eqnarray*}
A simple cut-off function $h(x)$ (Eq (\ref{cut-off-formula}) in Section \ref{sec:Trigonometric_Interpolation_of_General_Functions} is proposed and an enhanced trigonometric interpolation method is formulated in  Algorithm \ref{alg_textenion} in Section \ref{sec:Trigonometric_Interpolation_of_General_Functions}. The output $\hat{f}_M(x)$ interpolates $f(x)|_{[s,e]}$.  Some test results show that $\hat f$ has high degree accuracy as shown in Section \ref{subsec:numericalperformance_on_period_functions}.  We also demonstrate numerical evidences that the error of $\hat{f}_M(x)$ likely exhibits ``local property", i.e. error at a point tends not to propagate and cause significant compounding error at other place, which is not the case for polynomial-based approximations as shown in Section \ref{application_ode_general}. As such, the performance of $\hat {f}(x)$ is likely better than what is concluded in Theorem \ref{conv_k_new}. Details can be found in Section \ref{subsec:error_pattern}.  Algorithm \ref{alg_textenion} is applied to estimate integrals and solve linear/non-linear ordinary differential equation (ODE) as outlined in Algorithm \ref{alg_ode} and \ref{alg:ode_general} in Section \ref{sec:Applications}. The test results show that it outperforms Trapezoid/Simpson method to copy with integral and standard Runge-Kutta algorithm in handling ODE.

\section{The proof of Theorem \ref{main_thm}, \ref{converence_coef} and \ref{conv_k_new}} \label{sec:The_Development_of_the_Main_Theorems}
This section is used to set stage for the framework to be built and prove three theorems introduced in Section \ref{sec:main}. It starts with reviewing some relevant identities and developing certain equations, and then proves each of covered theorems in three subsections.  

In this section,  $f$ denote a $2b$-periodic function with $K+1$ derivative bounded by $D_{K+1}$.  $C(D_{K+1})$ denote a generic constant that depends on $D_{K+1}$ and its value may change on different situations.  
%We then investigate the properties of the interpolant, including estimations of coefficients, convergence of the interpolant and its derivative as well as anti-derivative.   

\subsection{Preliminary Algebraic Tools}
The classic Abel's transform (\ref{abel}) corresponds to integration by parts in the theory of integration \cite{Zygmund}, and plays a key rule in the error analysis of this section. For any two sequences of numbers $\{\alpha_i, \gamma_i\}_{i=0}^{n-1}$,    
\begin{equation}\label{abel}
	\sum^{k-1}_{i=0} \alpha_i\gamma_i = \alpha_{k-1} \Gamma_{k-1} - \sum_{i=0}^{k-2}(\alpha_{i+1}-\alpha_i)\Gamma_i, \qquad 1\le k\le n, 
\end{equation}
where $\Gamma_i = \sum_{j=0}^i\gamma_j$.  Throughout this paper, for any sequences with $n$ elements,  we always treat them as periodic sequences with period $n$, i.e $\alpha_l=\alpha_{k}$ and $\gamma_l=\gamma_{k}$ if $l=k$ mod $(n)$.  Throughout discussion,  we might modify index range of a summation without further reminding as follows.
%\begin{equation*}\label{period_series}
\[
\sum_{i=0}^{n-1} \alpha_i\gamma_i = \sum_{i=k}^{n-1+k} \alpha_i\gamma_i.
\]
%\end{equation*}
Recall that, for a positive integer $k$, $k$-th forward difference is defined inductively by
\[
\Delta_1 \alpha_i:=\alpha_{i+1}-\alpha_i,  \quad \Delta_k \alpha_i =\Delta_{k-1}(\Delta_1\alpha_i).
\]
One can derive inductively 
\[
\Delta_k \alpha_i = \alpha_{k+i} - k\alpha_{k+i-1} \cdots +  (-1)^{j} C^j_k \alpha_{k+i-j}  \cdots  + (-1)^k  \alpha_{i}, 
\]
where $C^j_k$ is  $j$-th coefficient of binomial polynomial $(1+x)^k$. It is clear that $\{ \Delta_k\alpha_i \}_{i\in \mathbb{Z}}$ is periodic with same period $n$ as  $\{\alpha_i \}_{-\infty <i<\infty}$, and we have following simple but important identity for a periodic sequence. 
\begin{equation}\label{sum_delta_ai}
	\sum_{i=0}^{n-1}\Delta_k\alpha_i = 0.
\end{equation}
As a special case where $\Gamma_{n-1}=0$,  Eq (\ref{abel}) is reduced to
\begin{equation}\label{abel_special}
	\sum^{n-1}_{i=0} \alpha_i\gamma_i =-\sum_{i=0}^{n-1}\Delta_1\alpha_i\Gamma_i.
\end{equation}
Eq (\ref{abel_special}) is equivalent to cancellation of boundary terms occurred in integration by parts with periodic functions, and it plays a key rule in derivation of estimation (\ref{C_{K+1}_1}). 

Adapt notations in Section \ref{main_thm}, and note that $k$-th forward difference of $f$ at any given point $x$ is defined inductively by
\[
\Delta^1_{\lambda}[f](x) = f(x+\lambda) -f(x), \quad  \Delta^k_{\lambda}[f] (x) = \Delta^{k-1}_{\lambda}[f] ( \Delta^1_{\lambda}[f](x)).
\]
One can verify
\[
\Delta^k_{\lambda}[f] (x) =\sum_{0\le j \le k} (-1)^{j} C^j_k f(x+ (k-j)\lambda).
\]  
For any integer $p\le k$, let
\[
H(p,k) = \sum_{0\le j\le k} (-1)^j C_k^j j^p.
\]
It is not hard to prove by induction that 
\begin{eqnarray*}
	H(p,k)=\left\{\begin{array}{cc}
		k! & p=k, \\
		0 & p<k. 
	\end{array}\right.
\end{eqnarray*}
Applying Taylor expansion to each item in $\Delta^k_{\lambda}[f] (x)$ at $x$, there exists $\xi \in [x, x+k\lambda]$ such that
\begin{equation}\label{deltakyj}
	\Delta^k_{\lambda}[f] (x) = f^{(k)}(x)  \lambda^k  +C_k(x) \lambda^{k},  
\end{equation}
where $C_k(x)$ is bounded and $\lim_{\lambda\to 0}C_k(x)=0$ if $f^{(k)}$ exists and is bounded. 

Recall following trigonometric identities 
\begin{eqnarray}
	\sum_{j=0}^{n-1} \sin (jx) &=& \frac12\cot\frac x2 - \frac12\cot\frac x2 \cos nx - \frac12\sin nx \label{sin_sum},\\
	\sum_{j=0}^{n-1} \cos(jx) &=& \frac12 + \frac12 \cot\frac x2 \sin nx - \frac12 \cos nx.   \label{cos_identity}
\end{eqnarray}
for any $x$ such that $x/\pi$ is not an integer.  Let $n\ge 2$ be an integer.  By Eq. (\ref{sin_sum}) and (\ref{cos_identity}) with $x=\frac{2\pi k}{n}$, we obtain
\begin{equation}\label{identity_cos_sin_sum}
	\sum_{j=0}^{n-1} cos\frac{2\pi jk}{n} = \sum_{j=0}^{n-1} sin\frac{2\pi jk}{n} = 0.
\end{equation}
By taking $x=\frac{2\pi k}{2n}$, we have
\begin{eqnarray}\label{identity_cos_sum}
	\sum_{j=0}^{n-1} cos\frac{2\pi jk}{2n}=\left\{\begin{array}{cc}
		n & \mbox{if $k=0 mod(2n)$,}  \\
		1 & \mbox{else if $k$ is old,} \\
		0 & \mbox{else if $k$ is even,} 
	\end{array}\right.
\end{eqnarray}
and
\begin{eqnarray}\label{identity_sin_sum}
	\sum_{j=0}^{n-1} \sin\frac{2\pi jk}{2n}=\left\{\begin{array}{cc}
		\cot\frac{\pi k}{2n} & \mbox{if $k$ is odd,} \\
		0 & \mbox{if $k$ is even.} 
	\end{array}\right.
\end{eqnarray}

If $x/\pi$ is not an integer,  replacing $x$ by $2x$ and  $n$ by $2n$ in Eq. (\ref{cos_identity}) respectively, we obtain
\begin{eqnarray}
	\sum_{j=0}^{n-1} \cos(2jx) &=& \frac12 + \frac12 \cot x \sin 2nx - \frac12 \cos 2nx,   \label{cos_identity_2n_2x} \\
	\sum_{j=0}^{2n-1} \cos(jx) &=& \frac12 + \frac12 \cot\frac x2 \sin 2nx - \frac12 \cos 2nx.   \label{cos_identity_2n}
\end{eqnarray}
Subtracting Eq. (\ref{cos_identity_2n_2x}) from Eq. (\ref{cos_identity_2n}) implies
\[
\sum_{j=0}^{n-1} \cos((2j+1)x) =\frac12 (\cot\frac x2 -\cot x) \sin 2nx.
\]
Plugging $x=\frac{2\pi k}{2n}$, we obtain 
\begin{eqnarray}\label{trig_odd_sum}
	\sum_{j=0}^{n-1} cos\frac{2\pi (2j+1) k}{2n}=\left\{\begin{array}{cc}
		0 & \mbox{if $k\neq 0  mod(n)$,}  \\
		n & \mbox{ if $k/n$ is even,} \\
		-n & \mbox{if $k/n$ is even.} 
	\end{array}\right.
\end{eqnarray}
Apply derivative on both sides of (\ref{cos_identity}), we obtain
\begin{eqnarray*}
	\sum_{j=0}^{n-1} j \sin jx &=&\frac14 \sin nx \csc^2\frac x2 -\frac n2 \sin nx -\frac n2 \cos nx \cot \frac x2.
	%\sum_{j=0}^{n-1} j^2 \cos jx &=& \frac {n^2} 2 (\sin nx \cot \frac x2 - \cos nx) + (\frac n2 \cos nx - \frac14 \sin nx \cot \frac x2)   \csc^2\frac x2   \nonumber
\end{eqnarray*}
Plugging $x=\frac{2\pi k}{2n}$ for $0<k<2n$ to above equations, we have 
\begin{eqnarray}
	\sum_{j=0}^{n-1} j \sin\frac{2\pi j k}{2n} &=& (-1)^{k+1}\frac n2 \cot\frac{\pi k}{2n} \label{key_idenity}.
	%\sum_{j=0}^{n-1} j^2 \cos \frac{2\pi j k}{2n} &=& (-1)^k \frac n2 \cot^2\frac{\pi k}{2n}  -\frac{(n-1)n}{2}(-1)^k \label{key_idenity_order2}
\end{eqnarray}
\subsection{The proof of Theorem \ref{main_thm}}\label{subsec:proof_main_thm}
Following observation is the key in development of Theorem \ref{main_thm}, whose proof can be found in Appendix \ref{pf_keylemma}.  
\begin{Lem}\label{keylemma} Adapt the notations in Section \ref{sec:main}.
	\begin{itemize} 	
		\item If $f(x)$ is even, define
		\begin{eqnarray}
			\tilde{f}_M(x) &=& \sum_{j=0}^{M-1}\tilde{A}_j \cos\frac{j\pi x}b,  \label{fhat_even}  \nonumber\\
			\tilde{A}_0 &:=&  \frac{2}N\sum_{k=0}^{N-1} y_k,  \qquad  \tilde{A}_j = a_j, \quad 1\le j <M. \nonumber 
		\end{eqnarray}	
		Let $\tilde{y}_l=\tilde{f}_M(x_l)$ for $0\le l < N$. Then 
		\begin{eqnarray}
			\tilde{y}_{2k} - y_{2k} &=&\frac1M\sum^{M-1}_{j=0}y_{2j+1},  \quad 0\le k < M,  \label{eq_keylemma_a_even}\\
			\tilde{y}_{2k+1}-y_{2k+1}&=&\frac1M\sum^{M-1}_{j=0}y_{2j},  \quad 0\le k < M. \label{eq_keylemma_a_odd}
		\end{eqnarray}
		\item If $f(x)$ is odd, $\tilde{f}_M$ fits to all nodes, i.e. 
		\begin{equation}\label{eq_keylemma_b}
			\tilde{y}_k - y_{k} = 0, \quad 0\le k < N. 
		\end{equation}
	\end{itemize}
\end{Lem}
With Eq (\ref{eq_keylemma_a_even}), (\ref{eq_keylemma_a_odd}) and (\ref{eq_keylemma_b}),  we are ready to prove Theorem \ref{main_thm}. 

\begin{proof} %The proof of Theorem \ref{main_thm}.  
	\begin{enumerate}
		\item  Let $f(x)$ be even.  By definition,
		\begin{equation}\label{aA0}
			\tilde{A}_0-a_0 = \frac1M\sum^{M-1}_{j=0}y_{2j+1},
		\end{equation}
		and therefore by (\ref{eq_keylemma_a_even}) and (\ref{aA0}), we have
		\[
		f_M(x_{2k}) - y_{2k} = \tilde{f}_M(x_{2k}) +a_0-\tilde{A}_0 -  y_{2k}= \tilde{y}_{2k} -\frac1M\sum^{M-1}_{j=0}y_{2j+1} -y_{2k} = 0.
		\]
		Similarly, by (\ref{eq_keylemma_a_odd}) and (\ref{aA0}), we have
		\[
		f_M(x_{2k+1}) - y_{2k+1} = \tilde{y}_{2k+1}-\tilde{A}_0+a_0-y_{2k+1}= \frac{2}N\sum_{j=0}^{N-1}(-1)^jy_{j}.
		\]
		To show uniqueness, assume that Eq (\ref{error_even_new}) and (\ref{error_odd_new}) hold, which implies 	for $0\le k<M$
		\begin{eqnarray}
			y_{2k}&=&\sum_{j=0}^{M-1}a_j (-1)^j \cos\frac{2\pi j (2k)}N  , \label{uniquness_even}\\
			(y_{2k+1} +\epsilon_M)  &=&\sum_{j=0}^{M-1}a_j (-1)^j \cos\frac{2\pi j (2k+1)}N. \label{uniquness_odd}
		\end{eqnarray}   
		Take summation on both sides of Eq (\ref{uniquness_even}) over $k$,  we obtain 	
		\[
		a_0=\frac1M\sum_{j=0}^{M-1}y_{2j}.
		\]
		For $0<l<M$,  Eq (\ref{uniquness_even})-(\ref{uniquness_odd}) imply 
		\begin{eqnarray*}
			y_{2k} \cos\frac{2\pi(2k)l}N  &=&\sum_{j=0}^{M-1}a_j (-1)^j \cos\frac{2\pi j (2k)}N  \cos\frac{2\pi(2k)l}N , \\
			(y_{2k+1} +\epsilon_M) \cos\frac{2\pi(2k+1)l}N &=&\sum_{j=0}^{M-1}a_j (-1)^j \cos\frac{2\pi j (2k+1)}N  \cos\frac{2\pi(2k+1)l}N.
		\end{eqnarray*}  
		Note that for fixed  $0< l<M$,
		\[
		\epsilon_M \sum_{k=0}^{M-1}\cos\frac{2\pi(2k+1)l}N =\epsilon_M\sum_{k=0}^{N-1}\cos\frac{2\pi kl}N -\epsilon_M\sum_{k=0}^{M-1}\cos\frac{2\pi (2k) l}N = 0.
		\]
		Hence
		\begin{eqnarray*}
			\sum_{k=0}^{N-1} y_k \cos\frac{2\pi k l}N &=& \sum_{j=0}^{M-1}\sum_{k=0}^{N-1}a_j (-1)^j \cos\frac{2\pi j k}N  \cos\frac{2\pi kl}N\\
			&=& \frac12 \sum_{j=0}^{M-1}\sum_{k=0}^{N-1}a_j (-1)^j (\cos\frac{2\pi k (j+l) }N +  \cos\frac{2\pi k (j-l)}N)\\
			&=& \frac N 2\sum_{j=0}^{M-1}\delta_{j,l}a_j(-1)^j  = a_l(-1)^l \frac N 2 
		\end{eqnarray*}
		which implies $a_l=\tilde{A}_l $ ($0<l<M$) as required. 
		
		\item Similarly, one can prove Theorem \ref{main_thm} in case that $f(x)$ is odd. 
	\end{enumerate}
\end{proof}
\subsection{The proof of Theorem \ref{converence_coef}}\label{subsec:proof_converence_coef}
We aim to prove Theorem \ref{converence_coef} in this subsection.  Classic Abel transform plays a similar tool as integration by part to derive estimations in discrete case.  Let us start to estimate following quantities for given positive integer pair $(l,k)$ with $l \le k$, 
\begin{eqnarray*}
	\phi_{l,k}&:=&\sum_{m=0}^{N-1} \Delta_{k} y_{m-k} \cos\frac{2\pi ml}N = \sum_{m=0}^{N-1} \Delta^{\lambda}_{k}[f] (x_{m-k})  \cos\frac{2\pi ml}N,\\
	\psi_{l,k}&:=&\sum_{m=0}^{N-1} \Delta_{k} y_{m-k} \sin\frac{2\pi ml}N = \sum_{m=0}^{N-1} \Delta^{\lambda}_{k}[f] (x_{m-k})  \sin\frac{2\pi ml}N.
\end{eqnarray*}
By Eq. (\ref{sum_delta_ai}),
\[
\phi_{l,k} = 2\sum_{m=0}^{N-1} \Delta^{\lambda}_{k}[f] (x_{m-k})   \cos^2\frac{\pi ml}n.
\]
Let $\Phi=\max_{x\in [-b,b]} \Delta^{\lambda}_{k}[f](x)$, we obtain
\[
\phi_{l,k} \le 2\Phi \sum_{m=0}^{N-1} \cos^2\frac{\pi ml}N = \Phi \sum_{m=0}^{N-1}(\cos\frac{2\pi ml}N +1)= \Phi N.
\]
Similarly,  we have $\phi_{l,k} \ge \phi N$ with $\phi=\min_{x\in [-b,b]} \Delta^{\lambda}_{k}[f](x)$. Applying same argument to $\psi_{l,k}$,  we conclude that there exist $\xi_{l,k}, \theta_{l,k} \in [-b, b]$ such that
\begin{equation}\label{phi_psi}
	\phi_{l,k} =  \Delta^{\lambda}_{k}[f](\xi_{l,k})N,\quad \psi_{l,k} = \Delta^{\lambda}_{k}[f](\theta_{l,k})N .
\end{equation}
We are now ready to prove Theorem \ref{converence_coef}.

\begin{proof} 
	By Eq. (\ref{abel_special}), (\ref{sin_sum}) and (\ref{cos_identity}), we can estimate $a_l$ for $0<l< M$.  Note $\sum_{m=0}^{N-1}\Delta y_m=0$,  we have
	\begin{eqnarray}
		(-1)^l \frac N2  a_l &=&-\sum_{m=0}^{N-1}\Delta y_m \sum_{j=0}^m \cos\frac{2\pi jl}N \nonumber\\
		&=&\frac12\sum_{m=0}^{N-1}\Delta y_m  (\cos\frac{2\pi l(m+1)}N  -  \frac12\cot\frac{\pi l}N \sin\frac{2\pi l(m+1)}N) \nonumber\\
		%&=&\sum_{m=0}^{N-1}\Delta y_m \cos\frac{2\pi l(m+1)}N -\cot\frac{\pi l}N \sum_{m=0}^{N-1}\Delta y_m\sin\frac{2\pi l(m+1)}N \nonumber\\
		&=&\frac12\sum_{m=0}^{N-1}\Delta y_{m-1} \cos\frac{2\pi l m}N -\frac12\cot\frac{\pi l}N \sum_{m=0}^{N-1}\Delta y_{m-1}\sin\frac{2\pi l m}N  \nonumber\\
		&=&\frac12 \phi_{l,1}-\frac12\cot\frac{\pi l}N \psi_{1,l} \label{al_1}  \nonumber
	\end{eqnarray} 
	Using Eq (\ref{abel_special}) $K$ more times, denote $w:=\cot\frac{\pi l}N$, we obtain
	\begin{eqnarray*}
		(-1)^l N2^{K}  a_l &=& \phi_{l,K+1} - C_{K+1}^1 w \psi_{l,K+1} -  C_{K+1}^2 w^2  \phi_{l,K+1}\\
		&+& C_{K+1}^3 w^3 \psi_{l,{K+1}} + C_{K+1}^4 w^4 \phi_{l,{K+1}} + \cdots\\
		&=& I_{\phi} - I_{\psi},
	\end{eqnarray*}
	where
	\begin{eqnarray*}
		I_{\phi} &=& \phi_{l,{K+1}} (1- C^2_{K+1} w^2 + C^4_{K+1} w^4 + \dots)\\
		&=& \frac {\phi_{l,{K+1}}}2 ((1+iw)^{K+1} +(1-iw)^{K+1}) = \frac{  \cos(\frac{\pi } 2 - \frac{\pi l}{N})(K+1)}{\sin^{K+1}\frac{\pi l}{N}} \phi_{l,{K+1}},
	\end{eqnarray*}
	and
	\begin{eqnarray*}
		I_{\psi} &=& \psi_{l,{K+1}} (C^1_{K+1} w- C^3_{K+1} w^3 + C^5_{K+1} w^5 + \dots)\\
		&=& \frac {\psi_{l,{K+1}}}{2i} ((1+iw)^{K+1} - (1-iw)^{K+1}) = \frac{  \sin(\frac{\pi } 2 - \frac{\pi l}{N})(K+1)}{\sin^{K+1}\frac{\pi l}{N}} \psi_{l,{K+1}}.
	\end{eqnarray*}
	By (\ref{phi_psi}) there exist $\xi_{l,{K+1}}, \theta_{l,{K+1}} \in [-b, b]$ such that
	\begin{eqnarray*}\label{a_l}
		a_l &=&  \frac{(-1)^l }{2^{K}\sin^{K+1}\frac{\pi l}{N}} (  \Delta^{\lambda}_{{K+1}}[f](\xi_{l,{K+1}}) \cos(\frac{\pi } 2 
		- \frac{\pi l}{N})(K+1)  \\
		&-& \Delta^{\lambda}_{{K+1}}[f](\theta_{l,{K+1}}) \sin(\frac{\pi }2-\frac{\pi l}{N})(K+1) ).
	\end{eqnarray*}
	Plugging  (\ref{deltakyj}) to above equation, we obtain
	\begin{equation}\label{C_{K+1}_1}
		|a_l|  \le \frac{C(D_{K+1})}{N^{K+1}\sin^{K+1}\frac{\pi l}{N}},
	\end{equation}
	where $ C_{{K+1},1}$ is a bounded constant depending on $D_{K+1}$. 
\end{proof}

\subsection{The proof of Theorem \ref{conv_k_new}}\label{subsec:proof_conv_k_new}
This section is mainly used to prove Theorem \ref{conv_k_new} with $f$ is even.  Same argument can be applied in parallel if $f$ is odd. 

We first develop connection between interpolant $f_M(x)$ and $f_{2M}$, which are based on $2M$ and $4M$ nodes by Eq. (\ref{x_grid_N}) respectively. In the case of $N=2M$,  define  
\begin{eqnarray}
	(-1)^l\bar{A}^N_l  &=& \frac{1}{N}\sum^{N-1}_{j=0}y_j \cos \frac{2\pi jl}N, \quad 0\le l < N. \label{barA_l}
\end{eqnarray}
Note that $\bar{A}^N_l$ are symmetric in the sense
\begin{equation}\label{sym_A_B}
	\bar{A}^N_l = \bar{A}^N_{N-l},  \qquad l= 1, \dots, N-1.
\end{equation}
Similarly, $\{a_j\}_{0\le j <M}$  in Eq. (\ref{a0})-(\ref{aj_even}) will be denoted by $\{a_j^N\}_{0\le j <M}$. Recall
\[
a_j^N = 2\bar A^N_j, \qquad 1\le j < M=N/2.
\]
By this convention,  $\{\bar{A}^{2N}_j\}_{0\le j <2N}$ and $\{a^{2N}_j\}_{0\le j <N}$ denote associated quantities with $2N$ equispaced nodes, i.e. $\lambda=\frac{2b}{2N}$.  Following lemma is the key observation for convergence analysis in this section.
\begin{Lem}\label{lemma-key-pattern} Let $a_j^N, a_j^{2N}$ be the coefficients of $f_M(x)$ and $f_{2M}(x)$ respectively, then
	\begin{equation}\label{a_N_2N}
		a_j^N = a_j^{2N} + a^{2N}_{N-j}, \qquad  1\le j < M.
	\end{equation}	
\end{Lem}

\begin{proof} Let $\bar{A}_j^N, \bar{A}_j^{2N}$ be defined by Eq. (\ref{barA_l}). Eq (\ref{a_N_2N}) is equivalent to
	\begin{equation}\label{rel_N_2N}
		\bar{A}_j^N = \bar{A}_j^{2N} + \bar A^{2N}_{N-j}, \qquad 0\le j< N.
	\end{equation}
	For $0\le j <N$, 
	\begin{eqnarray}
		(-1)^j\bar{A}_j^{2N}&=& \frac{1}{2N}\sum^{N-1}_{s=0}y_{2s} \cos \frac{2\pi s j}N + \frac{1}{2N}\sum^{N-1}_{s=0}y_{2s+1} \cos \frac{2\pi (2s+1)j}{2N} \nonumber\\
		&=& (-1)^j \frac12 \bar{A}_j^{N} + I_j \label{I_N_2N} ,
	\end{eqnarray}	 
	where $I_j :=\frac{1}{2N}\sum^{N-1}_{s=0}y_{2s+1} \cos \frac{2\pi (2s+1)j}{2N}$.  By Eq. (\ref{trig_odd_sum}) and (\ref{sym_A_B}), 
	\begin{eqnarray}
		I_j &=& \frac{1}{2N}\sum^{N-1}_{s=0} \cos \frac{2\pi (2s+1)j}{2N} \sum_{l=0}^{2N-1} (-1)^l A^{2N}_l \cos\frac{2\pi (2s+1)l}{2N}  \nonumber \\
		&=& \frac{1}{4N}\sum^{2N-1}_{l=0} (-1)^l A^{2N}_{l} \sum^{N-1}_{s=0}(\cos\frac{2\pi (2s+1)(l+j)}{2N} + \cos\frac{2\pi (2s+1)(l-j)}{2N} ) \nonumber\\
		&=& \frac{1}{4}\sum^{2N-1}_{l=0} (-1)^l A^{2N}_{l} (-\delta_{l+j=N} + \delta_{l+j=2N}  + \delta_{l-j=0} - \delta_{l-j=N}) \nonumber\\
		&=& \frac 12 (-1)^j (A^{2N}_j - A^{2N}_{N-j}), \nonumber
	\end{eqnarray}
	which, together with Eq (\ref{I_N_2N}), implies Eq. (\ref{rel_N_2N}).
\end{proof}

Let $f_{M}(x)$ be the interpolant using $N=2M$ nodes, define $\Delta_M(x)=f_M(x)-f_{2M}(x)$,  we have
\begin{equation}\label{delta_m_def}
	\Delta_M(x) = a_0^{N}-a_0^{2N} - a_M^{2N}\cos\frac{\pi M x}b + \sum_{M<j<N}a_j^{2N}(\cos \frac{(N-j)\pi x}{b}-\cos\frac{j\pi x}{b}).
\end{equation}
Notice $|a^{2N}_j|\le C(D_{K+1})/N^{K+1}$ for $M\le j<N$ and by Eq. (\ref{Euler-Maclaurin}), 
\[
|a_0^N - a_0^{N_p}| \le  |a_0^N-\frac {A_0}2| + |a_0^{N_p}-\frac {A_0}2| \le \frac{C(D_{K+1})}{N^K},
\]
hence
\[
|\Delta_N(x) | \le \frac{C(D_{K+1})}{N^{K}}.
\]
For a given $N=2^q$ and an integer $p\ge 0$, define $M_p=2^{p-1}N$, and  $f_{M_p}$ be the associated interpolant with $2M_p$ nodes and $\Delta_{M_p}(x)=f_{M_p}-f_{2M_p}$, we have
\begin{eqnarray}
	|f_M(x) - f_{M_p}(x)| &\le& \sum_{0\le r \le p}|f_{M_r}(x) - f_{M_{r+1}}(x)|\nonumber\\ 
	&\le& \frac{C(D_{K+1})}{N^{K}}   \sum_{0\le r \le p} \frac{1}{2^{(r-1)K}} \le \frac{C(D_{K+1})}{N^{K}}. \label{detal_M}
\end{eqnarray}
For integer $1\le k < K$, by (\ref{delta_m_def}), we obtain estimation on $k-th$ derivative of $\Delta_M(x)$ ,
\begin{eqnarray}
	\Delta^{(k)}_M(x) &\le& (\frac{\pi M}b)^k |a_M^{2N}| + \sum_{M<j<N} |a_j^{2N}|[(\frac{(N-j)\pi}{b})^k +  (\frac{j\pi}{b})^k]\nonumber\\
	&\le& \frac{C({D_{K+1}})}{N^{K-k}} \label{der_Delta_M}, \nonumber  
\end{eqnarray}
which implies
\begin{eqnarray}
	|f^{(k)}_M(x) - f^{(k)}_{M_p}(x)| &\le& \sum_{0\le r \le p}|f^{(k)}_{M_r}(x) - f^{(k)}_{M_{r+1}}(x)|\nonumber\\ 
	&\le& \frac{C(D_{K+1})}{N^{K-k}}   \sum_{0\le r \le p} \frac{1}{2^{(r-1)K}} \le \frac{C(D_{K+1})}{N^{K-k}}. \label{detal_M_2}
\end{eqnarray}
Estimations  (\ref{detal_M}) and (\ref{detal_M_2}) imply that $f_{M_p}$ and $f^{(k)}_{M_p}$ converge uniformly as $p\to \infty$.  It is clear that $f_{M_p}$ converges to $f(x)$ on the dense set $S=\cup_{p=0}^{\infty} S_p$ where
\[
S_p=\{ -b + 2k\frac{2b}{2M_p}, \quad 0\le k < M_p.\} 
\]
Therefore, $f_{M_p}$ converges to $f(x)$ and consequently  $f^{(k)}_{M_p}$ converges to $f^{(k)}(x)$ as $p\to \infty$. Applying $p\to\infty$ to Estimation (\ref{detal_M}) and (\ref{detal_M_2}), we obtain Estimation (\ref{detal_M_f}) and (\ref{der_detal_M_f}).

\section{Trigonometric Estimation of General Functions}\label{sec:Trigonometric_Interpolation_of_General_Functions}
This section is used to develop a trigonometric interpolation algorithm that can be applied to a non-periodic function $f$ over a bounded interval $[s,e]$. As such,  $f(x)$ shall denote a function whose $K+1$-th derivative $f^{(K+1)}(x)$ exists and is bounded for some $K\ge 1$, but need not be periodic in this paper. 

We can shift $f(x)$ by $s$ and then evenly extend it to $[s-e, e-s]$. Such direct extension deteriorates the smoothness at $0, \pm (e-s)$, and leads to a poor convergence performance as showed in Section \ref{sec:Numerical_Performance}.  To seek for a smooth periodic extension, we assume that $f$ can be extended smoothly such that $f^{(K+1)}$ exists and is bounded over $[s-\delta, e+\delta]$ for certain $\delta>0$ and leverage a cut-off smooth function $h(x)$ with following property: 
%One can use the following method to periodically extend $g$ by a cut-off function $h$ with the following properties 
\begin{eqnarray*}
	h(x)=\left\{\begin{array}{cc}
		1 & x\in [s,e], \\
		0 & \mbox{$x<s-\delta$ or $x>e+\delta$}. \\
	\end{array}\right.
\end{eqnarray*}
Such cut-off function can be constructed in different ways and we shall adopt one with a closed-form analytic expression as follows: 
\begin{equation}\label{cut-off-formula}
	h(x; r, s,e, \delta) = B(\frac{x-(s-\delta)}{\delta}, r)\times B(\frac{e+\delta-x}{\delta}, r), 
\end{equation}
where 
\[
B(x; r)=\frac{G(x; r)}{G(x; r) + G(1-x; r)}, \quad 
G(x; r)=\left\{\begin{array}{cc}
	e^{-\frac{r}{x^2}} & \mbox{$x>0$,}  \\
	0 & \mbox{$x\le 0$.} 
\end{array}\right.
\]
To see the effect of parameter $r$, Figure \ref{fig:cut_off_performance_5_5} plots left wing of cut-off functions with three scenarios $r\in \{ 0.1, 0.5, 1\}$ for $(s,e, \delta)=(-1,1, 1)$.  As $r$ increases from $0.1$ to $1$, $h$ takes more space in $x-axis$ to increase around $0$ to near $1$ and therefore change is less dramatically over the process, which is preferable.  On the other hand, smaller $r$ provides more spaces for $h(x)$ converges to $0$ and $1$, which is also preferable for the performance of $h$ around $s$ and $s-\delta$.  
\begin{figure}[H]\label{fig:cut_off_performance_5_5}
	\includegraphics[width=10cm]{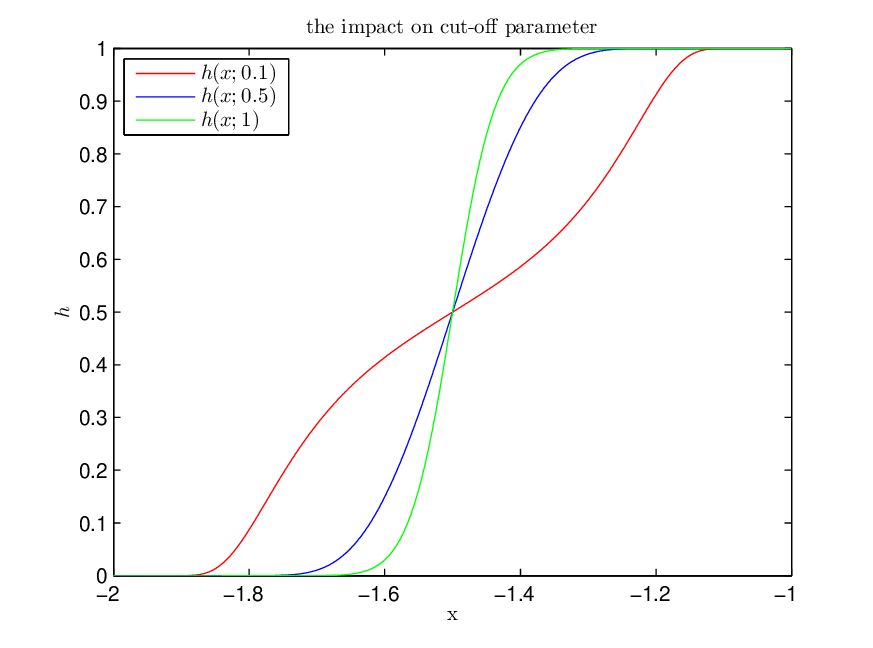}
	\caption{The graphs of $h_M$ over $[s-\delta, s]$ with $(s,e, \delta)=(-1,1, 1)$ and $M=2^8$ for three cases: $r=0.1$ (red),  $r=0.5$ (blue) and $r=1$ (green).  }
\end{figure}
Trigonometric expansion of $h$ can be useful in applications, and Table \ref{tab:impact_cutoff_parameter_p} shows max error in two cases with $M=2^q$ grid nodes, which recommends $r=0.5$ and we shall take it on all tests reported in this paper.  
\begin{table}[htbp]
	\caption{Max differences between $h_M$ and $h$ where $h_M$ is the trigonometric estimation of $h$ with $M=2^q$ grid nodes.}
	\begin{tabular}{rrr}
		\multicolumn{1}{l}{q} & \multicolumn{1}{l}{r} & \multicolumn{1}{l}{Max Error} \\ \hline\hline
		8     & 0.1   & 9.6E-07 \\ \hline
		8     & 0.5   & 1.5E-10 \\ \hline
		8     & 1     & 3.3E-08 \\ \hline
		10    & 0.1   & 3.4E-15 \\ \hline
		10    & 0.5   & 2.7E-15 \\ \hline
		10    & 1     & 2.9E-15 \\ \hline
	\end{tabular}%
	\label{tab:impact_cutoff_parameter_p}%
\end{table}%
With the cut-off $h(x)$, $h(x)f(x)$ can be smoothly extended to $[2s-e-3\delta, s-\delta]$ symmetrically with respect to vertical line $x=s-\delta$. The idea is demonstrated by an example where $f(x)=(x-2.5)^2$ for $x\in [2,3]$ with $(s,e,\delta)=2,3,1$, as shown in Figure \ref{fig:cut_off_extension_a}.
\begin{figure}[H]\label{fig:cut_off_extension_a}
	\includegraphics[width=10cm]{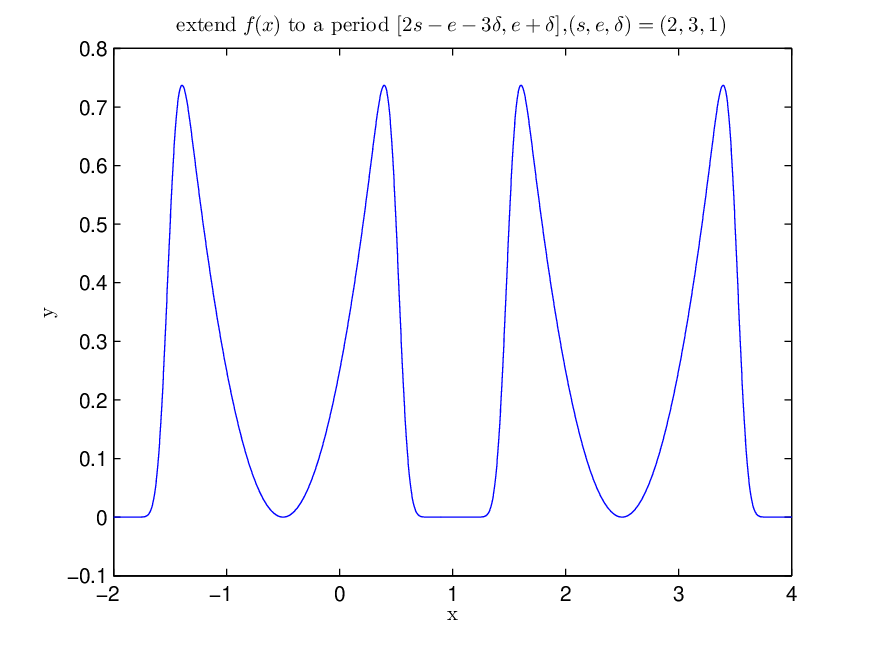}
	\caption{The graphs of the $fh$'s extension $(hf)_{ext}$ over a period $[2s-e-3\delta, e+\delta]$ with $(s,e,\delta)=(2,3,1)$.  The function is even after parallel shift left by $s-\delta=1$. Note that $fh=f=(x-2.5)^2$ over $[s,e]$ as expected.}
\end{figure}
The periodic extension can be summarized as follows:
\begin{Alg}\label{alg_textenion} 
	Let $f(x)$ be defined over $[s,e]$. 
	\begin{enumerate}
		\item Select integers $0<p<q$ such that $f(x)$ can be smoothly extended to $[s-\delta, e+\delta]$,  where
		\begin{eqnarray*}
			n&=&2^p, \quad  M=2^q,  \quad \lambda=\frac{e-s}n, \\	
			m&=&\frac{M-n}2, \qquad  \delta = m\lambda .
		\end{eqnarray*}
		\item Construct the cut-off function $h(x)$ with parameter $(s,e,\delta)$.
		\item Let
		\begin{equation}\label{ob}
			o=s-\delta, \quad b=e+\delta -o,
		\end{equation}
		and define $F(x):=h(x+o)f(x+o)$ for $x\in [0,b]$.
		\item Extend $F(x)$ evenly by $F(x)=F(-x)$ for $x\in [-b,0]$ \footnote{Alternatively, extend $F(x)$ oddly by $F(x)=-F(-x)$ for $x\in [-b,0]$ if odd trigonometric estimation is desired.}. It is clear that $F(x)$ can be treated as an periodic even function. 
		\item Define grid nodes by 
		\[
		x_j = -b + j\lambda, \quad j=0, 1, \cdots N-1, \quad N=2M,
		\]
		and apply them to construct trigonometric expansion $F_M$ by Theorem \ref{main_thm} \footnote{Alternatively, $F(x)= \sum_{1\le j < M} a_j \sin\frac{j\pi x}{b}$ if odd trigonometric interpolant is desired.} :
		\[
		F_M(x) = \sum_{0\le j < M} a_j \cos\frac{j\pi x}{b}.
		\]
		\item Let 
		%\begin{equation*}\label{fMext}
		\[
			\hat{f}_{M}(x)=F_M(x-o)=\sum_{0\le j < M} a_j \cos\frac{j\pi (x-o)}{b}. 
		\]
		%\end{equation*}
	\end{enumerate}
\end{Alg}
$\hat{f}$ will be used to denote the extended periodic function by Algorithm  \ref{alg_textenion} and $\hat{f}_M$ be the interpolant of $\hat{f}$ by Algorithm \ref{alg_textenion} in the rest of this paper.  Clearly, it inherits same smoothness as $f(x)$ does. We shall discuss its performance and applications in Section \ref{sec:Numerical_Performance} and  \ref{sec:Applications} respectively. 

Figure \ref{fig:cut_off_sample_2} compares $f=(x-2.5)^2$ and $\hat{f}_M$ over $[s-\delta, e+\delta]$ with $[s,e]=[2,3]$.  Note that $\hat{f}_M$ recovers $f$ over $[s,e]$ and approaches to $0$ around boundary points $s-\delta, e+\delta$ as expected. 
\begin{figure}[H]\label{fig:cut_off_sample_2}
	\centering
	\includegraphics[width=10cm]{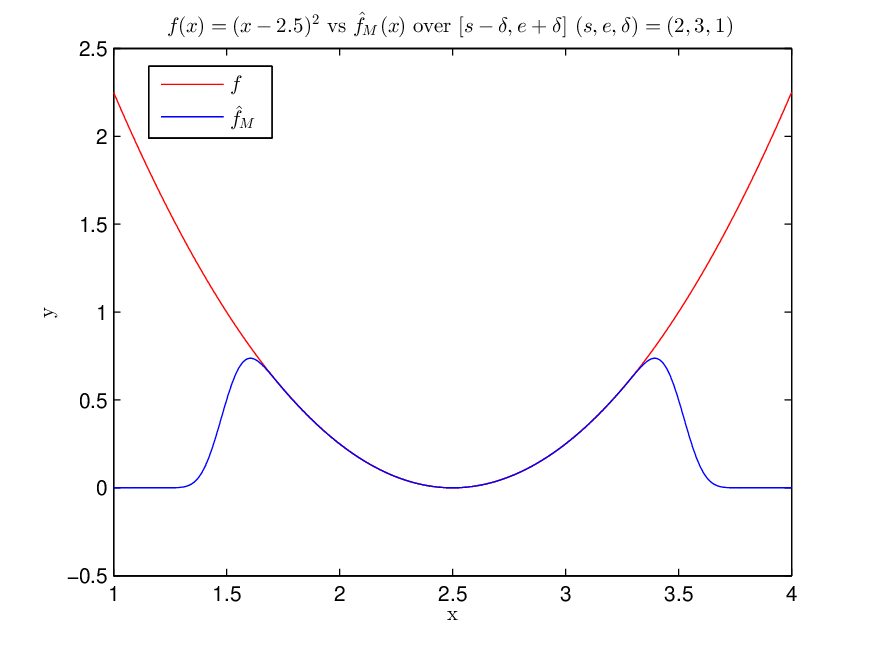}
	\caption{The graphs of $f(x)$ vs $\hat{f}_M(x)$ over $[e-\delta,s+\delta]$ with $(s,e, p, q, \delta)=(2,3,7,8,1)$. The figure is plotted by $2^{12}$ sample points.}
\end{figure}

\section{Numerical Performance}\label{sec:Numerical_Performance}
This section provides some numerical results to test  performance of $\hat f_M(x)$.  First, we test convergence performance of periodic function, and show that it is sensitive to smoothness of underlying function $f$ as expected.  Secondly,  we apply the enhanced algorithm to two sets of typical functions whose values can be highly oscillated and rapidly changed, and show that it does exhibit stable and accurate performance. 

\subsection{Numerical Performance on Periodic Functions}\label{subsec:numericalperformance_on_period_functions} Let $f$ be the even periodic function with period $2\pi$ and be defined as follows over $[-\pi, \pi]$:
%\begin{equation*}
\[
	f(x; d)|_{[-\pi, \pi]} = (1-(\frac x{\pi})^2)^d,
\]
%\end{equation*}
where $d$ can be $1,2$.   It is clear that $f(x;1)$ is not differentiable at $\pm \pi$ and $f(x;2)$ is $2$-th continuous differentiable.  We expect that interpolant of $f_M(x;2)$ has significant better performance than $f_M(x;1)$. 

Table \ref{tab:max_diff_fun_der} provides max errors under various settings on grid points. It confirms that performance is sensitive to number of grid points and especially degree of $f$'s smoothness as expected. 
\begin{table}[H]
	\caption{The max errors $|f_M(x)-f(x)|$ and  $|f'_M(x)-f'(x)|$ }
	\begin{tabular}{rrcc}
		\hline\hline
		\multicolumn{1}{r}{$d$} & \multicolumn{1}{r}{$M$} & \multicolumn{1}{c}{$\max(|f-f_M|)$}   & \multicolumn{1}{c}{$\max (|f'-f'_M|)$}  \\
		\hline
		1     & 16    & 2.52E-02  & 6.4E-01 \\
		1     & 64    & 6.02E-03 &  6.4E-01 \\
		1     & 256   & 1.48E-03 &  6.4E-01  \\
		1     & 1024  & 3.55E-04 &  6.4E-01 \\\hline
		2     & 16    & 4.29E-05  & 5.7E-04 \\
		2     & 64    & 5.62E-07  & 3.5E-05 \\
		2     & 256   & 8.26E-09  & 2.2E-06 \\
		2     & 1024  & 1.28E-10 & 1.4E-07 \\\hline
		\hline
	\end{tabular}%
	\label{tab:max_diff_fun_der}%
\end{table}%

\subsection{Numerical Performance on General Functions}\label{subsec:numericalperformance_on_general_functions}
Let $\hat{f}$ be the smooth extension described in Section \ref{sec:Trigonometric_Interpolation_of_General_Functions} and $\hat{f}_M$ be the trigonometric estimation of $\hat{f}$ by Algorithm \ref{alg_textenion}.  

We apply Algorithm \ref{alg_textenion} to following functions to test convergence performance. 
\begin{eqnarray*}\label{test_fun}
	y&=& \cos\theta x, \quad \theta =1, 10, 100,\\
	y &=& x^n,\qquad n = 4,8,10.
\end{eqnarray*}
For any estimation $\tilde w(x)$ of a function $w(x)$,  define  max error in log space over a subset $S$ of the domain of $w$ by
\begin{equation}\label{error_metric}
	E(w)=\max_{x \in S}\{ log_{10} |\tilde{w}(x)-w(x)| \}.
\end{equation}

Let $f_M$ be the interpolant of direct periodic extension of $f$ without using cut-off function. Table \ref{tab:smooth_impact_cos} shows such errors up to second derivatives of both $f_M$ and $\hat{f}_M$ for the two functions with the parameters $(s,e, p,q, \delta) = (-1,1,7,8,1)$, and the subset $S=\{s+k\frac{e-s}{2^{12}}$, $0\le k \le 2^{12}\}$. Note that we deliberately make $S$ larger than the set of grid nodes used in Theorem \ref{main_thm} to test that $\hat{f}_M$ converges at non-grid nodes.  The performance of $\hat{f}_M$ is stable across all the test scenarios and max estimation errors are small, and $f_M(x)$ generates significant errors, especially on derivatives, which confirms the impact of smoothness on performance as expected. 
\begin{table}[H]
	\small
	\caption{Max error in log space (EL) with parameters $(s,e, p, q, \delta) = (-1,1,7,8,1)$.}
	\begin{tabular}{lrrrrrr}
		\multicolumn{1}{l}{$para$} & \multicolumn{1}{l}{$EL(f_M)$} & \multicolumn{1}{l}{$EL(\hat{f}_M)$} & \multicolumn{1}{l}{$EL(f'_M)$} & \multicolumn{1}{l}{$EL(\hat{f}_M')$} & \multicolumn{1}{l}{$EL(f''_M)$} & \multicolumn{1}{l}{$EL(\hat{f}_M'')$}  \\ \hline \hline
		\multicolumn{7}{c}{$f=\cos\theta x$} \\ \hline \hline
		$\theta=1$     & -3.2  & -14.7 & 0.0   & -13.1 & 2.8   & -10.7 \\ \hline 
		$\theta=10$    & -2.4  & -14.8 & -0.3  & -14.2 & 1.6   & -11.8 \\ \hline 
		$\theta=100$   & -1.4  & -14.0 & -0.3  & -14.0 & 0.6   & -11.9 \\ \hline 
		\multicolumn{7}{c}{$f=x^n$} \\ \hline \hline
		$n=4$     & 0.4   & -14.8 & 0.0   & -13.6 & 2.4   & -11.1 \\ \hline
		$n=8$     & 0.3   & -14.3 & 0.0   & -13.1 & 2.0   & -10.6 \\ \hline
		$n=10$    & 0.3   & -14.0 & 0.0   & -12.9 & 1.9   & -10.4 \\ \hline \hline
	\end{tabular}%
	\label{tab:smooth_impact_cos}%
\end{table}
\subsection{Discussion on Error behavior of $\hat f_M(x)$}\label{subsec:error_pattern}
The error analysis in Section \ref{sec:The_Development_of_the_Main_Theorems} is based upon two fundamental identities (\ref{sin_sum}) and (\ref{cos_identity}), which might imply a feature of a trigonometric polynomial, i.e. estimation error tends to be small due to certain cancellation when target function is smooth.  If further assuming that $f$ bears some symmetry (even or odd) and the set of grid points is selected by equispaced nodes,  then performance with high accuracy rate is expected since cancellation effect becomes more propounded as evidenced by subsequently derived identities (\ref{identity_cos_sin_sum}), (\ref{identity_cos_sum}), (\ref{identity_sin_sum}), (\ref{trig_odd_sum}), and (\ref{key_idenity}),  which play essential rules in the error analysis.  

Note that it is hard, if not impossible, to fully exploit this feature in theoretic analysis. We provide some supplementary numerical evidence in this section. We look into error behaviors of $\hat f_M(x)$.  Intuitively, as mentioned above,   the errors should exhibit ``local property", i.e. error at a point should not propagate and cause large compounding error at other points.  This is not the case when polynomial based method is used as shown in Section \ref{application_ode_general}.  We conduct relevant tests on four basic functions power, exponential, sin and cos function.  Figure \ref{fig:consecutive_diff} shows the normalized differences of consecutive errors defined by
$\frac 1{\max |f(x)|}\{\hat {f}_M(x_i) - f(x_i) - (\hat {f}_M(x_{i-1}) -f(x_{i-1}))\}.$
A clear sawtooth pattern is shown in Figure \ref{fig:consecutive_diff} for all test cases. Magnitude of error keep reasonable stable, but goes in alternative directions, a strong sign that error is not accumulating, but canceled with each other when variable $x$ moves around. Same phenomena has been observed on the estimation error about solution of non-linear ODE discussed in Remark \ref{rem:test_2_performance}.    As such,  performance of trigonometric interpolant $\hat {f}(x)$ is likely better than what is concluded in Section \ref{sec:main_results}.
\begin{figure}[H]\label{fig:consecutive_diff}
	\centering
	\includegraphics[width=5cm]{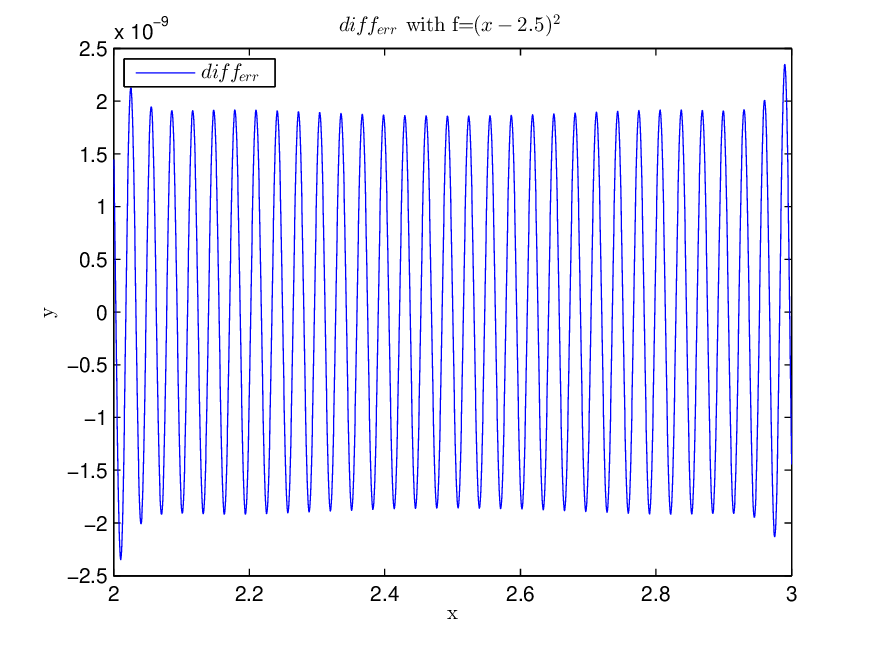}
	\includegraphics[width=5cm]{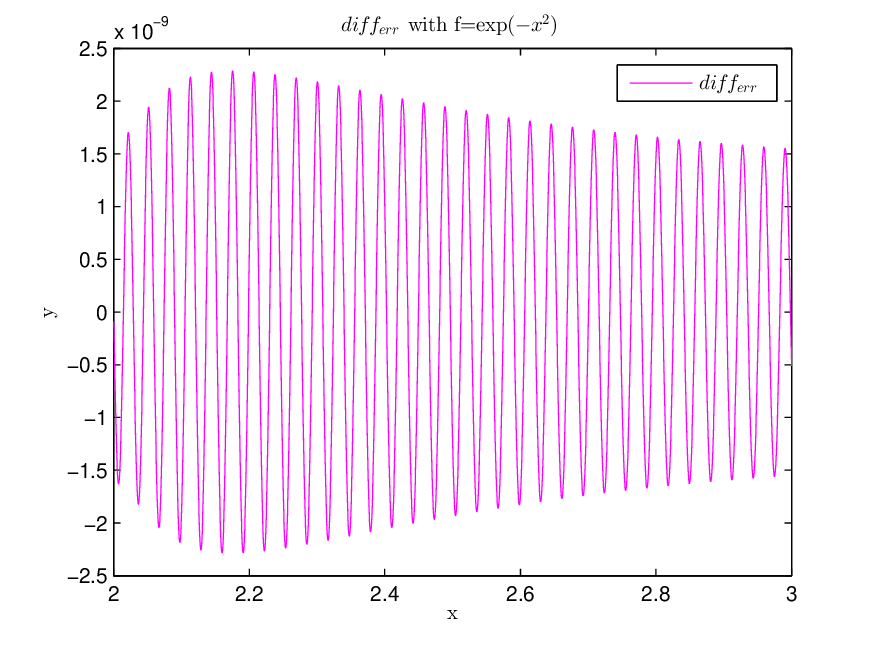}
	\includegraphics[width=5cm]{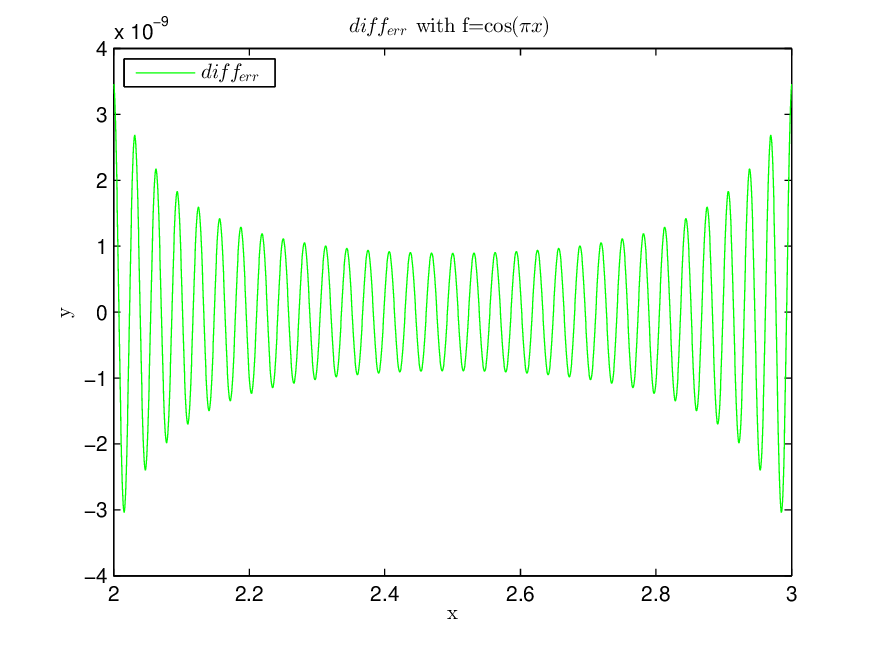}
	\includegraphics[width=5cm]{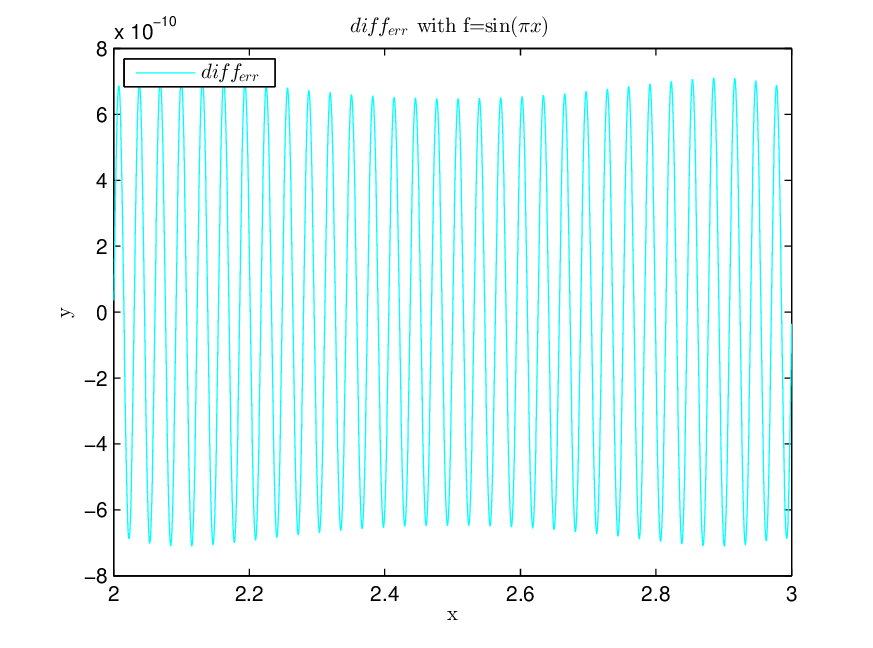}
	\caption{Plots of the normalized difference of consecutive error at grid points with $\lambda=1/2^6$ with $[s,e]=[2,3]$}
\end{figure}

\section{Applications}\label{sec:Applications}
The trigonometric estimation $\hat{f}_{M}(x)$ of a general function $f(x)$ by Algorithm  \ref{alg_textenion} can be used to develop numerical computational algorithms in various subjects that we shall discuss in subsequent papers. In this section,  we provide  two examples.  First,  $\hat{f}_M$ is used to solve an integral with a general integrand. The test results show that  the new method exhibits stable performance and outperforms that of popular algorithms such as Trapezoid and Simpson methods, especially when the integrands are highly oscillated. Furthermore, we provide an algorithm to solve first order linear ODE since the solution can always be represented as an integral.  Secondly,  we demonstrate how the algorithm can be effectively used in optimization by solving a general first order ODE.

The notations in Section \ref{sec:Trigonometric_Interpolation_of_General_Functions} will be adapted in this Section.

\subsection{Numerical Solution of a General Integral}
Let $f$ be a $K+1$ differentiable function over $[s-\delta,e+\delta]$ with $\delta>0$.  We aim to solve an integral  $\int^{e}_{s} f(u)du$.  Applying the approximation $\hat{f}$ described in Algorithm \ref{alg_textenion},
the integral can be estimated by
\begin{equation}\label{FM_cutoff}
	\int^{e}_{s} f(u)du  \approx a_0 (e-s) +  \sum_{1\le j <M} \frac{ba_j }{j\pi} (\sin\frac{j\pi (e-o)}{b} -\sin\frac{j\pi (s-o)}{b}),
\end{equation}
where $o, b$ are defined in Eq. (\ref{ob}). 

To see the performance, we conduct two sets of tests with integral $x^n$ and $\cos\theta x$
\[
\int^1_{-1} x^n dx, \qquad \int^1_{-1}\cos\theta x dx.
\] 
For each test, we use the metric defined in Eq. (\ref{error_metric}) to compare the performance of three methods: Estimation (\ref{FM_cutoff}), Trapezoid as well as Simpson, and display the results in Table \ref{tab:integration_power} and \ref{tab:integration_cos}.
\begin{table}[htbp]
\begin{tabular}{ccccc}
		\multicolumn{1}{l}{n} & \multicolumn{1}{l}{Trig. Estimation} & \multicolumn{1}{l}{Trapezoid Method} & \multicolumn{1}{l}{Simpson Method}  \\ \hline \hline
		4     & -15.5 & -5.0  & -10.2 \\ \hline
		8     & -14.3 & -4.7  & -9.1 \\ \hline
		10    & -14.3 & -4.6  & -8.7 \\ \hline
\end{tabular}%
\caption{The max errors in log space with three methods:  Estimation Eq (\ref{FM_cutoff}), Trapezoid, and Simpson. Note that $2^q$ denotes the number of grid points in interpolation algorithm with parameter $(s,e,p,q,\delta)=(-1,1,7,8,1)$.}
\label{tab:integration_power}%
\end{table}%
\begin{table}[htbp]
	\caption{Similar comparison as shown in Table \ref{tab:integration_power} for integral  $\cos \theta x$}
	\begin{tabular}{ccccc}
		\multicolumn{1}{l}{$\theta$} & \multicolumn{1}{l}{Approx. (\ref{FM_cutoff})} & \multicolumn{1}{l}{Trapezoid Method} & \multicolumn{1}{l}{Simpson Method} \\ \hline \hline
		1     & -15.4 & -5.7  & -11.7 \\ \hline
		10    & -16.4 & -4.9  & -8.9 \\ \hline
		100   & -16.8 & -3.9  & -5.9 \\ \hline
	\end{tabular}%
	\label{tab:integration_cos}%
\end{table}%
One can see that Simpson outperforms significantly Trapezoid as expected.  When $n$ and $\theta$ increases,  the performance of Simpson and Trapezoid deteriorates as expected since the integrals change more dramatically, especially for $\cos\theta x$. Estimation (\ref{FM_cutoff}) turns out to be more robust to handle rapid changes and high oscillation of function values as long as the integrand is sufficient smooth. 

Compared to numerical algorithms like Simpson and Trapezoid,  another advantage of Estimation (\ref{FM_cutoff}) is  that it provides an attractive closed-form approximation of anti derivatives.  As such,  it can be used to estimate
the solution of a general first order linear ordinary differential equation (ODE)
\begin{eqnarray}
	y'(x)+P(x)y&=&Q(x), \quad x\in [s, e] \label{ODE_1}\\
	y(x_0) &=& y_0, \label{ODE_2}
\end{eqnarray}
where $P(x), Q(x)$ are two functions over $[s, e]$ with first order continuous derivative. 
It is well-known that the solution can be expressed as 
%\begin{equation*}\label{ode_d1}
\[
	y(x) = e^{-\int^x_{x_0}P(u)du }(y_0 + \int^x_{x_0} Q(u) e^{\int^u_{x_0}P(v)dv}du),
\]
%\end{equation*}
which can be reformulated as 
\begin{eqnarray*}
	y(x) &=& \frac{y_0 + G(x)}{I(x)}, \\
	I(x)&=&e^{\int^x_{x_0}P(u)du}, \quad  G(x) = \int^x_{x_0} I(u)Q(u)du.
\end{eqnarray*}

We summarize the algorithm as following steps.  
\begin{Alg}\label{alg_ode} For a given ODE (\ref{ODE_1})-(\ref{ODE_2}), define $(p,q,\delta)$ as in Algorithm \ref{alg_textenion} such that $P,Q$ can be smoothly extended to $[s-\delta, e+\delta]$.   
	\begin{enumerate}
		\item Construct the cut-off function $h(x)$ with parameter $(s,e,\delta)$ by Eq. \ref{cut-off-formula}.
		\item Construct estimation $\hat{P}_M$ of $P(x)$ based on Algorithm \ref{alg_textenion} using grid nodes $\{x_k\}_{0\le k<2M}$
		%\begin{equation*}
		\[
			{x_k} = o-b + k\lambda, \quad k=0, \dots, 2M-1,  \quad M=2^q,
		\]
		%\end{equation*}
		and obtain
		\begin{equation}\label{pM_cutoff}
			\hat{P}_{M}= \sum_{0\le j <M} a_j \cos\frac{j\pi (x-o)}{b}, \quad x\in [s,e],
		\end{equation}
		where $o, b$ are defined in Eq. (\ref{ob}).
		\item Estimate $I(x_k)$, denoted by $\tilde{I}(x_k)$, based on Eq. (\ref{pM_cutoff})  and  (\ref{FM_cutoff}).
		%\item Calculate $\int^x_{x_0}\hat{P}(u)du$ using (\ref{FM_cutoff}) at grid points $\{x_j\}_{0\le j <M}$;
		\item Construct $\hat{IQ}(x)_M$ with the grid points $\{(x_k \tilde{I}(x_k)Q(x_k))\}$  based on Algorithm \ref{alg_textenion}.
		\item Use $\hat{IQ}(x)_M$ to estimate $G(x)$ by Eq. (\ref{FM_cutoff}).
		\item Apply estimated  $G(x)$ and $I(x)$ to construct grid values for $y$.
		\item Apply Algorithm \ref{alg_textenion} to estimate the solution $y$.
	\end{enumerate}
\end{Alg}
\begin{Rem}
	\begin{enumerate}
		\item Eq (\ref{FM_cutoff}) can be carried out by FFT at grid points as in Theorem \ref{main_thm}, making Algorithm \ref{alg_ode} computationally efficient. 
		\item $I(x)$ can be too large to be handed properly when $P$ keep positive and $x$ is significantly away from $x_0$.  In this case,  one should break the range into small pieces and find the solutions on each of them separately. 
	\end{enumerate}
\end{Rem}

The algorithm is tested with $P(x)=Q(x)=x^2$ over $[s,e]=[1,3]$ with initial value $y(1)=y_0$. The close-formed solution is 
\[
y(x) = (y_0-1)e^{\frac{1-x^3}3} + 1.
\]
With $M=2^8$,  The max error in magnitude is around $1.8\times 10^{-7}$ for all three initial conditions, suggesting that the algorithm is robust to generate accurate solutions.  

\begin{figure}[H]\label{fig:ode_numerica}
	\centering
	\includegraphics[width=6cm]{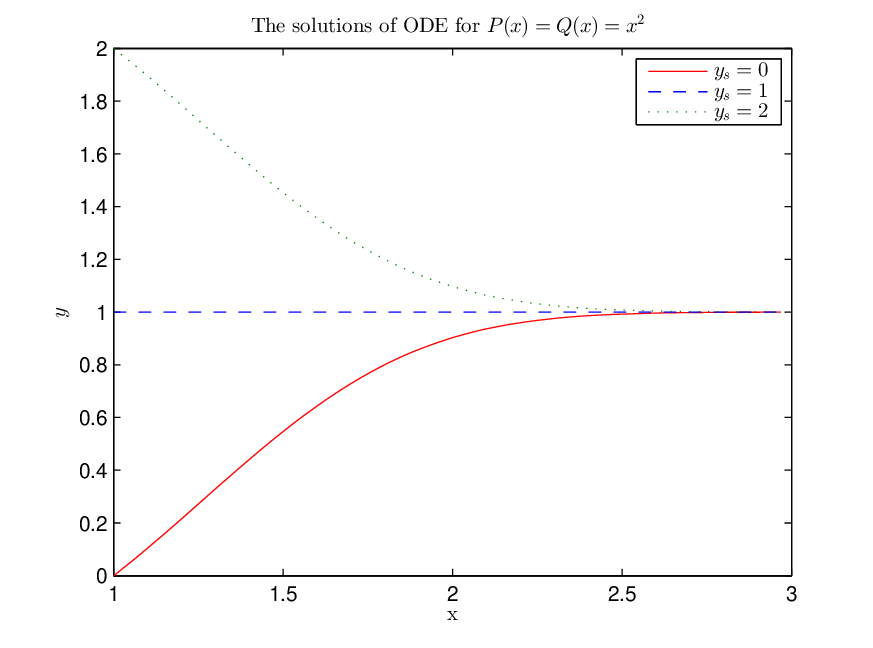}
	\caption{The solutions of ODE (\ref{ODE_1})-(\ref{ODE_2}) by Algorithm \ref{alg_ode} with three initial values at $s$:  $(0, 1, 2)$ with $q=8$.}
\end{figure}

\subsection{Numerical Solution of First Order Differential Equation}\label{application_ode_general}
In this subsection, we aim to develop an algorithm on numerical solution of a general first order ODE:
\begin{eqnarray}
	y'(x) &=& f(x,y),  \quad x\in [s,e] \label{ode_general_eq_v2}\\
	y(s) &=& \xi, \label{ode_general_init_v2}
\end{eqnarray}
where $f(x,y)$ is continuously differential on the range $[s-\delta, e+\delta]\times R^1$. It is well-known that there is a unique solution of Eq. (\ref{ode_general_eq_v2})-(\ref{ode_general_init_v2}) \cite{Boyce_Dirprima}.

Replacing $f(x,y)$ by $f(x+o,y)$ if needed, we assume that $o:=s-\delta=0$ and denote $h(x)$ as the cut-off function in Section \ref{sec:Trigonometric_Interpolation_of_General_Functions} such that
\begin{eqnarray*}
	h(x)=\left\{\begin{array}{cc}
		1 & s\le x \le e, \\
		0 &  \mbox{ $x\le s-\delta$ or  $x\ge e+\delta$}.  \\
	\end{array}\right.
\end{eqnarray*}
Apply $h$ to extend $f(x,y)$ as follows \footnote{We use $u$ to denote the periodic extension of $y$.}: 
\begin{eqnarray}\label{ode_general_F_v2}
	F(x,u)=\left\{\begin{array}{cc}
		f(x,u)h(x) & \mbox{if $x \in [0,b]$},  \\
		-f(-x,u)h(-x) & \mbox{if $x \in [-b,0]$}. 
	\end{array}\right.
\end{eqnarray}
We shall search numerical solution for the extended ODE:
\begin{eqnarray}
	u'(x) &=& F(x,u),  \quad x\in [-b,b] \label{ode_general_f_v2},\\
	u(s) &=& \xi. \label{ode_general_f_init_v2}
\end{eqnarray}
Since $(u(x)-u(-x))' \equiv 0$, $u(x)$ is even and its derivative $z(x):=u'(x)$ is odd. It is clear that $u$ can be smoothly extended to even periodic function with period $2b$ and $u(x)|_{[s,e]}$ solves Eq. (\ref{ode_general_eq_v2})- (\ref{ode_general_init_v2}).  Let  $\{(x_k, z_k)\}_{0\le k <N}$ be a grid set of $z(x)$:
\begin{eqnarray}
	x_k &=& -b + \frac{2b }N k, \qquad k=0, 1, \cdots, N-1, \label{grid_x_ode}\\
	z_0&=&0, \quad z_{N-k} = -z_k,  \qquad 1\le k <M, \label{grid_z_ode}
\end{eqnarray}
and 
\begin{equation}\label{z_M}
	z_M(x) = \sum_{0\le j<M}b_j \sin \frac{j\pi x}{b} 
\end{equation}
be the interpolant of $z(x)$ with 
\begin{equation}\label{b_j_ext}
	b_j=\frac{2}N\sum_{k=0}^{N-1} (-1)^j z_k \sin\frac{2\pi j k}{N}=\frac{4}N\sum_{k=0}^{M-1} (-1)^j z_k \sin\frac{2\pi j k}{N}, \quad 0\le j < M. 
\end{equation}
It is clear 
%\begin{equation*}\label{b_j_z_k}
\[
	\frac{\partial b_j}{\partial z_k} = \frac{4}N (-1)^j  \sin\frac{2\pi j k}{N}, \quad 0\le j,k < M. 
\]
%\end{equation*}
$u$ can be approximated based on Eq. (\ref{z_M})  by
\begin{equation}\label{tildeu}
	\tilde{u}_M(x ) = \sum_{0\le j <M}  a_j\cos  \frac{j\pi x}{b}, \quad a_j= -\frac{bb_j}{j\pi }, \quad 1\le j <M,
\end{equation}
and $a_0$ can be solved by the initial condition $u(-s)=u(x_{m+n})=\xi$ 
\[
a_0 = \xi-\sum_{1\le j <M} (-1)^j a_j\cos  \frac{2\pi  j (m+n) }{N}.
\]

Let $0_M$ be the $M$-dim zero vector and define $\frac{1}0:=0$. The following notations will be adopted in the rest of this subsection. 
\begin{eqnarray*}
	u_k &=& \tilde{u}_M(x_k), \quad F_k=F(x_k, u_k), \quad DF_k =\frac {\partial F}{\partial u} (x_{k}, u_{k}),\\
	Z&=& \{z_k\}_{0\le k<M}, \quad U= \{u_k\}_{0\le k<M}, \quad F = \{F_k\}_{0\le k<M}, \quad DF = \{DF_k\}_{0\le k<M},\\
	J &=& [0, 1, \frac12, \dots, \frac1{M-1}, 0_M], \\
	\Phi&=&[ \frac{1}{j} \cos \frac{2\pi j (m+n) }N ]_{j=0}^{M-1}, \quad \Phi_N = [\Phi, 0_M], \\
	\Psi&=&\{ (z_j-F_j)DF_j \}_{j=0}^{M-1},  \quad \Psi_N = [\Psi, 0_M],  \quad \quad I = sum(\Psi). 
\end{eqnarray*}
ODE (\ref{ode_general_f_v2})-(\ref{ode_general_f_init_v2}) can be solved by minimizing the following error function:  
\begin{equation}\label{target_phi}
	\phi(z_{0}, z_{1}, ...,z_{M-1}) = \frac1{2M}\sum_{0\le k < M} (z_k-F_k)^2.  
\end{equation}

We need an effective way to calculate its gradient $\frac{\partial \phi}{\partial Z}$ when $M$, the number of variables of $\phi$, is not small.
\begin{equation}\label{gradient_v2}
	M\frac {\partial \phi}{\partial z_t} = (z_t-F_t) - \sum_{0\le k<M} (z_k-F_k) DF_k \frac {\partial u_k}{\partial z_t}.
\end{equation}

To copy with $\frac{\partial U}{\partial Z}$ in Eq. (\ref{gradient_v2}), we need express $U$ in term of $Z$.    By Eq (\ref{b_j_ext}) and $z_0=z_M=0$, we obtain for $0\le k <N$
\begin{eqnarray}
	u_k &=&  a_0 - \sum_{0\le j <M} (-1)^j \frac{bb_j}{j\pi }  \cos  \frac{2 \pi jk }{N} \nonumber\\
	&=&  a_0 - \frac{2b}{\pi N} \sum_{0\le j <M} \frac{1}{j}  \cos \frac{2 \pi jk }{N} \sum_{0\le l <N} z_l \sin\frac{2\pi j l}{N} \label{uk_1}\\
	&=& a_0 - \frac{2b }{\pi N}\sum_{0\le l <N} z_l \sum_{0\le j <M}  \frac1j \cos  \frac{2 \pi jk }{N}   \sin\frac{2\pi j l}{N} \nonumber\\
	&=& a_0 - \frac{4b }{\pi N}\sum_{0 \le l <M} z_l \sum_{0\le j <M}  \frac1j \cos  \frac{2 \pi jk }{N}   \sin\frac{2\pi j l}{N}. \label{uk_2} 
\end{eqnarray}
The last step is due to $z_{l}\sin\frac{2\pi j l}{N}=z_{N-1}\sin\frac{2\pi j (N-l)}{N}$. 
One can rewrite (\ref{uk_1}) in term of ifft as follows:
\begin{equation}\label{U}
	U = a_0 - \frac{2bN}{\pi} Re\{ifft(J\circ Im\{ifft(Z)\})\},
\end{equation}
where $\circ$ denotes the Hadamard product, which applies the element-wise product to two metrics of same dimension.    
$a_0$ in (\ref{uk_2}) can be further interpreted by $Z$ as follows:
\begin{eqnarray}
	a_0 &=& \xi-\sum_{1\le j <M} (-1)^j a_j\cos  \frac{2\pi  j (m+n) }{N}  \nonumber\\
	&=& \xi+  \frac{b}{\pi }\sum_{1\le j <M} (-1)^j \frac{b_j}{j }\cos \frac{2\pi  j (m+n) }{N}   \nonumber\\
	&=& \xi+  \frac{2b}{\pi N }\sum_{1\le j <M}  \frac{1}{j } \cos \frac{2\pi  j (m+n) }{N}  \sum_{k=0}^{N-1}  z_k \sin\frac{2\pi j k}{N}, \label{b_j_z_k_2}
	%&=& \xi + \frac{2b}{\pi} \Phi_N\cdot Im(ifft(\{z_k\})) \nonumber
\end{eqnarray}
which implies 
\begin{eqnarray}
	\frac{\partial a_0}{\partial z_k}  &=&   \frac{4b}{\pi  N}  \sum_{0\le j <M} \frac{1}{j}  \cos  \frac{2\pi j (m+n)\pi }{N} \sin\frac{2\pi j k}{N}, \quad 0\le k<M. \label{a_0_z_k}
\end{eqnarray}
Combining (\ref{uk_2}) and (\ref{a_0_z_k}), we obtain 
\begin{eqnarray}	
	\frac{\partial u_k}{\partial z_t} &=&  \frac{4b}{\pi  N}  \sum_{0\le j <M} \frac{1}{j}  \cos  \frac{2\pi  j (m+n) }{N} \sin\frac{2\pi j t}{N} \nonumber\\
	&-&\frac{4b }{\pi N} \sum_{0\le j <M}  \frac 1j \cos  \frac{2 \pi jk }{N}   \sin\frac{2\pi j t}{N} \label{z_k_v2} \label{uz_2}.
\end{eqnarray}

We are ready to attack the non-trivial term in Eq. (\ref{gradient_v2}). Define 
\[
w_t :=\sum_{0\le k<M} (z_k-F_k) DF_k \frac {\partial u_k}{\partial z_t}.
\]
By (\ref{uz_2}),
\begin{eqnarray*}
	w_t &=& \frac{4b}{\pi N }  \sum_{0\le j <M} \frac 1j \cos  \frac{2\pi  j (m+n)\pi }{N}   \sin\frac{2\pi j t}{N} \sum_{0\le k<M} (z_k-F_k) DF_k \nonumber \\
	&-& \frac{4b }{\pi N }  \sum_{0\le j <M} \frac 1j \sin\frac{2\pi j t}{N}   \sum_{0\le k<M} (z_k-F_k) DF_k  \cos  \frac{2 \pi jk }{N} \label{3t_2_1}\\
	&=& \frac{4b }{\pi N}  \sum_{0\le j <M} \Phi_j   \sin\frac{2\pi j t}{N} \sum_{0\le k<M} (z_k-F_k) DF_k \nonumber \\
	&-& \frac{4b }{\pi N}  \sum_{0\le j <M} J_j \sin\frac{2\pi j t}{N}   \sum_{0\le k<M} (z_k-F_k) DF_k  \cos  \frac{2 \pi jk }{N}. \label{3t_2}
\end{eqnarray*}
The gradient vector (\ref{gradient_v2}) can be formulated by FFT as follows:
\begin{eqnarray}
	W &=& \frac{4 b I}{\pi} Im(ifft(\Phi_N)) -\frac{4\pi N}b Im \{ifft (J\circ Re[ifft(\Psi_N)])\}, \nonumber\\
	\frac{\partial \phi}{\partial Z}&=& \frac{1}{M}(Z-F-W[0:M-1]). \label{phi_to_z}
\end{eqnarray}
One can implement the algorithm by following steps.
\begin{Alg}\label{alg:ode_general} %Numerical solution of first order ODE by trigonometric interpolation. 
	For a given ODE (\ref{ode_general_eq_v2})-(\ref{ode_general_init_v2}), 
	\begin{enumerate}
		\item Select $(p,q,\delta)$ as in Algorithm \ref{alg_textenion} such that $f(x,y)$ can be smoothly extended to $[s-\delta, e+\delta]\times R$.   
		\item Construct the cut-off function $h(x)$ with parameter $(s,e,\delta, r=0.5)$ by Eq. (\ref{cut-off-formula}). 
		\item Construct $F(x,u)$ by (\ref{ode_general_F_v2}).
		\item \label{alg:ode_general_init}  Construct initial values $\{u_i\}_{0\le i <N}$ at $\{x_i\}_{0\le i <N}$ defined by (\ref{grid_x_ode}) and calculate initial value $z(x_i)=F(x_i,u_i)$ in (\ref{grid_z_ode}) that is required by an optimization function.
		\item Apply an optimization function with the gradient function $\frac{\partial \phi}{\partial Z}$ formulated by Eq. (\ref{phi_to_z}).
		\item Apply the opt values of $Z$ returned by the optimization function in previous step to calculate required $U$ by (\ref{U}) and (\ref{b_j_z_k_2}) and return $U|_{[s,e]}$.
	\end{enumerate}
\end{Alg}
\begin{Rem}
	\begin{enumerate}
		\item Algorithm \ref{alg:ode_general} is expected to be efficient since the gradient of the target function can be carried out by $O(N\ln_2 N)$ operations. 
		\item  A standard difference method can be used to construct initial values $\{u_i\}_M^{N-1}$ in Step \ref{alg:ode_general_init} of Algorithm \ref{alg:ode_general}, which should be smoothed by the cut-off function and then be  extend evenly at the grid points over $[-b,0]$. 
	\end{enumerate}
\end{Rem}

In the rest of this subsection, we study the performance of Algorithm \ref{alg:ode_general}, labeled $intp$, by solving the following ODE
\begin{equation}\label{ex:ode_general}
	y'(x) = f(x,y) + xy +y ^2,  \qquad  y(1) = 0,
\end{equation}
where
\[
f(x,y) = \cos \theta x - \theta x \sin \theta x - xy - y^2, \theta \in \{ \frac{\pi}{2}, \frac{3\pi}{2} \}.
\]
The analytic solution is available as follows:
\[
y(x) = x\cos \theta x.
\]
We shall compare $intp$ to the classic Runge–Kutta method, labeled $rk4$, outlined in Eq (\ref{rk4-1})-(\ref{rk4-3}) \cite{ptvf}.
\begin{eqnarray}
	y_{n+1} &=& y_n + \frac{\lambda}{6}(k_1+2k_2+2k_3+k_4), \label{rk4-1}\\
	k_1 &=& f(x_n, y_n) ,\quad k_2 = f(x_n + \frac{\lambda}{2}, y_n +  \frac{\lambda}{2}k_1), \label{rk4-2}\\
	\quad k_2 &=& f(x_n + \frac{\lambda}{2}, y_n +  \frac{\lambda}{2}k_2), \quad k_2 = f(x_n + \lambda, y_n + \lambda k_3). \label{rk4-3}
\end{eqnarray}
It is well-known that local truncation error of $rk4$ is on the order of $O(\lambda^{5})$ and hence the total accumulated error is supposed to be $O(\lambda^{4})$.

In addition, we implement a benchmark method, labeled as $benc$, by adjusting $rk4$ in a ``cheating" way that $y_{n+1}$ is estimated by the true value $Y_{n}$ at $x_n$.  Mathematically, all $y_n$ in Eq (\ref{rk4-1})-(\ref{rk4-3}) is replaced by $Y_n$.  In this way, $benc$ prevents accumulating error and is supposed to be on the order of $O(\lambda^{5})$. 

The overall performance is reported in Table \ref{tab:ode_general_performance}, where the max magnitude of errors at grid nodes $\{x_j\}_{j=0}^{N-1}$ are shown under three methods.  In addition, to see the performance of $intp$ at non-grid nodes,  the max error is also reported under $intp_{g}$ where the max is taken over the error set obtained by applying identified $\tilde{u}_M$ (see Eq. (\ref{tildeu})) to the grid points with step size $\lambda/4$.  Table \ref{tab:ode_general_performance} also includes the value of target function Eq. (\ref{target_phi}) returned by Matlab function $fmincon$ under Column $opt_{err}$.  
\begin{table}[htbp]
	\caption{The max magnitudes of four sets of errors and the optimization error described above.  Algorithm \ref{alg:ode_general} is implemented with $p=6, q=7,\delta=1$.}
	\begin{tabular}{crrrrr}
		$\theta$ & \multicolumn{1}{c}{$intp$} & \multicolumn{1}{c}{$rk4$} & \multicolumn{1}{c}{$benc$} & \multicolumn{1}{c}{$intp_{g}$} & \multicolumn{1}{c}{$opt_{err}$}  \\ \hline\hline
		${\pi}/2$  & 3.2E-09 & 7.7E-07 & 3.0E-08 & 3.2E-09 & 3.2E-17  \\\hline
		${3\pi}/2$ & 4.8E-07 & 2.1E-03 & 1.1E-05 & 4.8E-07 & 1.0E-17  \\
	\end{tabular}%
	\label{tab:ode_general_performance}%
\end{table}%
We also look into changes of consecutive errors for three covered methods similar as we did in Section \ref{subsec:error_pattern}. 
Figure \ref{fig:max_error_analysis} plots such changes defined by 
\[
\{\hat {f}_M(x_i) - f(x_i) - (\hat {f}_M(x_{i-1}) -f(x_{i-1}))\}.
\] 
for $intp$ and $benc$. Figure \ref{fig:diff} compares same changes among three covered methods. We have the following comments.
\begin{Rem}\label{rem:test_2_performance}
	\begin{enumerate}
		\item Optimization process successfully converges to the desired interpolant $z_M(x)$ of the target function $z(x)=u'(x)$ for both scenarios as shown by $opt_{err}$ values in Table \ref{tab:ode_general_performance}.
		\item $intp$ is almost same as $intp_g$, suggesting that $\tilde {u}_M$ uniformly converges the target function $u$ with same accuracy as exhibited at grid points at ${x_j}_{0\le j<N}$. 
		\item The performances at scenario $\theta=\pi/2$ are significantly better than that at scenario $\theta=3\pi/2$ as expected since the target function $y$ by (\ref{ex:ode_general}) with small $\theta$ is much less volatile than with large $\theta$. 
		\item $intp$ outperforms $benc$ significantly not just by the measure on max error in Table \ref{tab:ode_general_performance}, but also its error is consistently smaller than $benc$'s as shown in Fig \ref{fig:max_error_analysis}. 
		\item\label{rem:test_2_performance_part5} $rk4$ has worst performance based on max error, especially for $\theta=3\pi/2$.  Fig. \ref{fig:diff} plots the difference of two consecutive errors for each of three methods.  As one can see,  errors from $rk4$ moves in one direction from certain point $x_j$ and leads to significant aggregated error at the end.  On the other hand,  the error of $intp$ moves in sawtooth around $0$, an error-correct sign, which makes $intp$ outperform $bech$, a method without error propagation. See Section \ref{subsec:error_pattern} for more discussions of error behaviors of $\hat f_M(x)$.
	\end{enumerate}
\end{Rem} 
\begin{figure}[H]\label{fig:max_error_analysis}
	\centering
	\includegraphics[width=5cm]{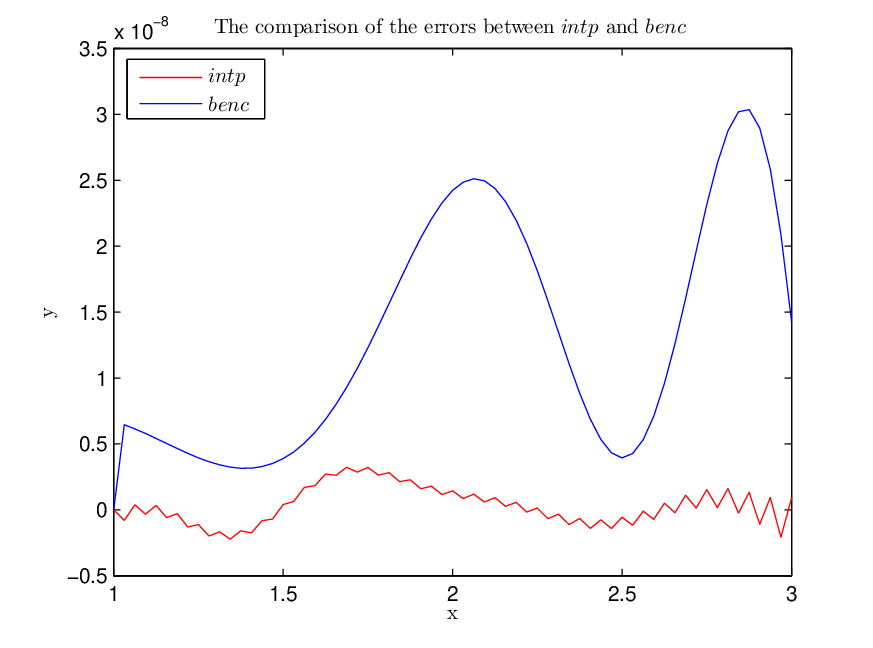}
	\includegraphics[width=5cm]{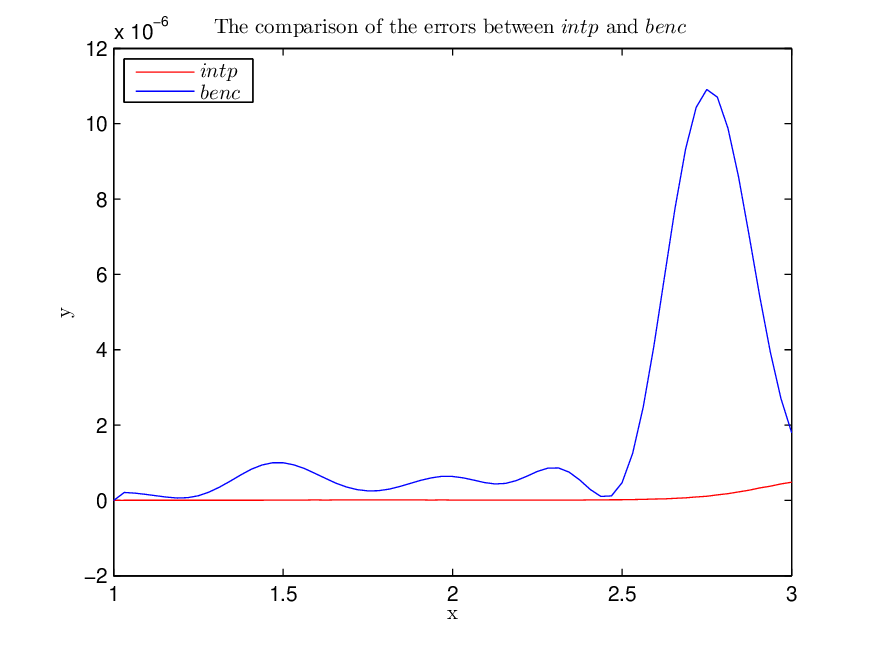}
	\caption{The comparison of consecutive errors between $intp$ and $benc$ for $\theta=\pi/2$ on the left and $\theta=3\pi/2$ on the right.} 
\end{figure}

\begin{figure}[H]\label{fig:diff}
	\centering
	\includegraphics[width=5cm]{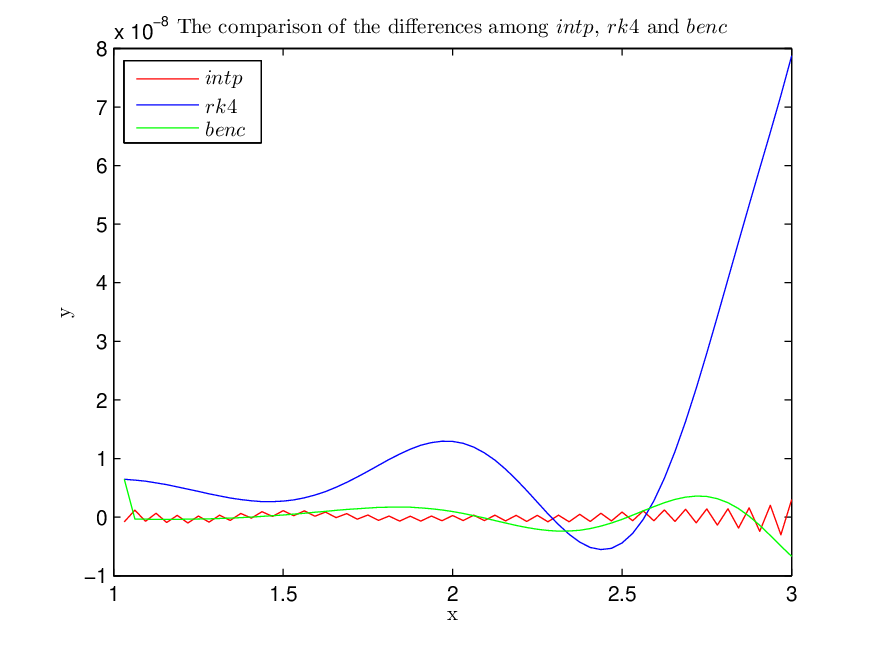}
	\includegraphics[width=5cm]{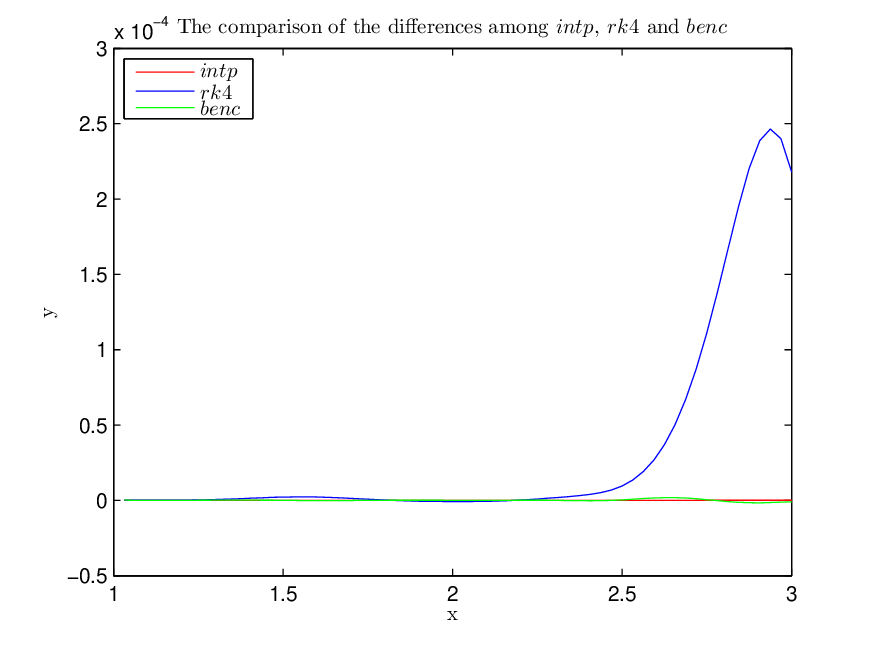}
	\caption{The comparison of changes of consecutive errors among $intp$, $benc$, $rk4$ for $\theta=\pi/2$ on the left panel and $\theta=3\pi/2$ on the right panel.}
\end{figure}
Fig. \ref{fig:uandy} plots the target $y$ and identified $\tilde u_M$ over $[0,b]$.  $\tilde u_M$ recovers $y$ over the range $[s,e]=[1,3]$. It becomes flat near boundaries and can be treated as even periodic function.      
\begin{figure}[H]\label{fig:uandy}
	\includegraphics[width=5cm]{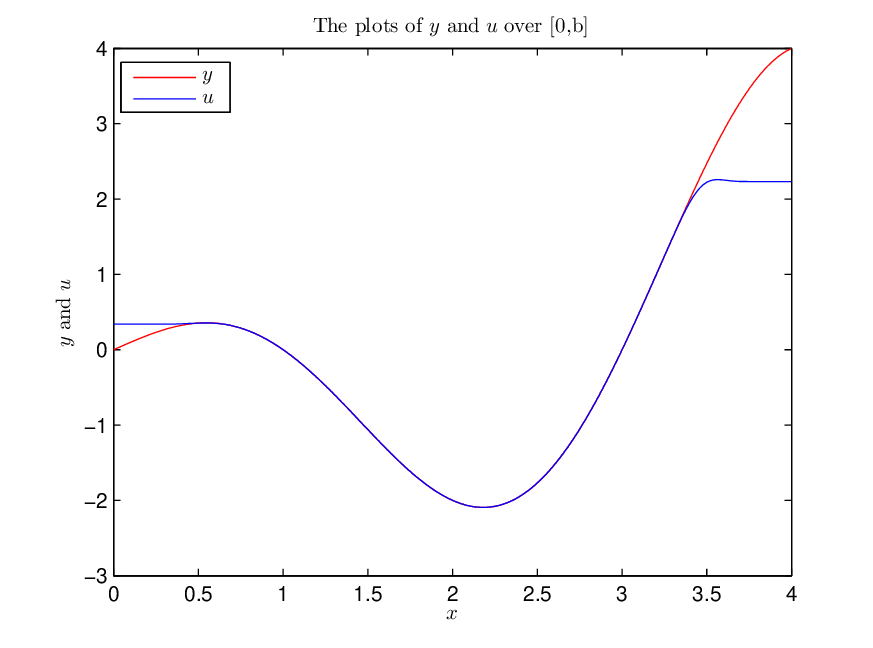}
	\includegraphics[width=5cm]{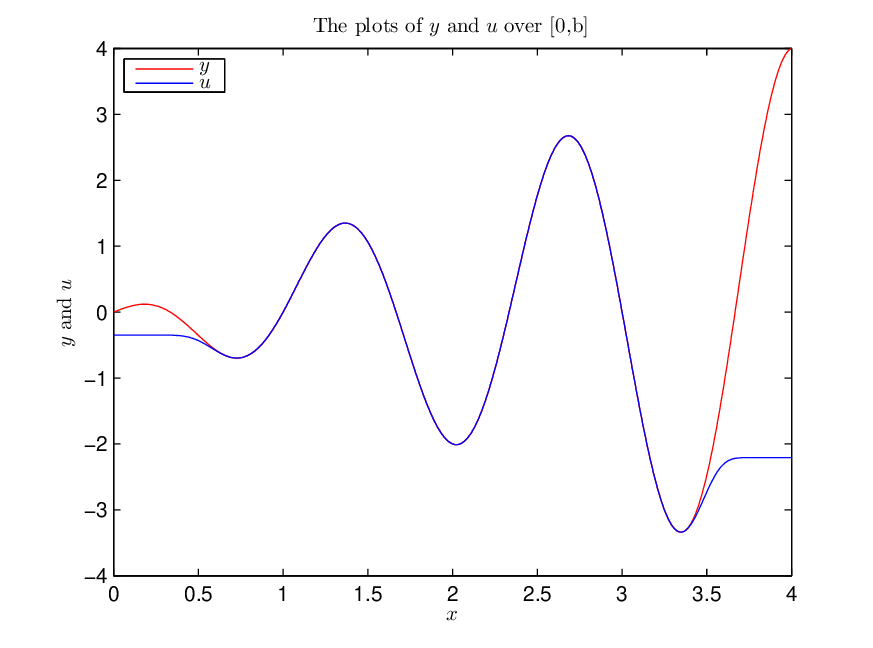}
	\caption{ Comparison between the target solution $y$ and trigonometric estimation $\tilde u_M$ over the half period $[0,b]=[0,4]$. Note that $\tilde u_M$ is supposed to approximate $y$ over $[s,e]=[2,3]$.}
\end{figure}

The optimization method of solving first order non-linear ODE can be extended to a high order non-linear ODE with flexible initial and boundary conditions. In \cite{zou_ODE}, we shall apply the trigonometric interpolation algorithm to solve second order non-linear ODE with various setting on initial and boundary conditions.  We use the same objective function, formulate the gradient vector in a similar way and pay extra attention to handle boundary and condition conditions.  Certain numerical testing shows consistent good performances as shown in this subsection. 	

\section{Conclusions}\label{sec:Conclusions}
In this paper, we have proposed a new trigonometric interpolation method and established convergent properties. The method improves an trigonometric interpolation algorithm in the literature such that it can better leverage Fast Fourier Transform (FFT) to enhance efficiency. The interpolant can be formulated in such way that the cancellation effects can be more profoundly leveraged for error analysis, which enables us not only to  improve the convergent rate of interpolant,  but establish similar uniform convergence for the derivatives of interpolants. We have further enhanced the method that can be applied to non-periodic functions defined on bounded interval. Numerical testing results confirm accurate performance of the algorithm.  As application, we demonstrate how it can be applied to estimate integrals and solve linear/non-linear ODE. The test results show that it outperforms Trapezoid/Simpson method to copy with integrals and standard Runge-Kutta algorithm to handle ODE. In addition,  we show numerical evidences that estimation error of the algorithm likely exhibits ``local property", i.e. error at a point tends not to propagate and result in significant compounding error at some other place, a remarkable advantage compared to polynomial-based approximations, suggesting the performance of the algorithm is likely better than what has be theoretically approved in this paper.

Considering the accurate performance and analytic attractiveness of the new trigonometric interpolation algorithm, especially it can be applied to non-periodic functions, we expect that it can be used in a wide spectrum.  We shall further study its application in subsequent papers \cite{zou_ODE}-\cite{zou_ODE_ide_volterra}. 
\section*{Acknowledgments}
This paper is dedicated to my wife Jian Jiao for her unwavering support to the family in last thirty plus years.

\appendix

\section{The proof of Lemma \ref{keylemma} }\label{pf_keylemma}
If $f(x)$ is even,  by definition,
\[
\hat{y}_{l} = \sum_{j=0}^{M-1}A'_j \cos\frac{j\pi x_l}b=\sum_{j=0}^{M-1}A'_j (-1)^j \cos\frac{2\pi j l}N,  \quad 0\le l < N.
\]
By Eq. (\ref{identity_cos_sum}), for terms with even index $l=2k$, we obtain  
\begin{eqnarray}
	\hat{y}_{2k} &=& \sum_{j=0}^{M-1} (-1)^j A'_j cos\frac{2\pi j (2k)}{N} \nonumber\\
	&=& \sum_{j=0}^{M-1} (-1)^j cos\frac{2\pi jk}{M} \frac{2}N\sum_{l=0}^{N-1} (-1)^j y_{l}cos\frac{2\pi jl}{N}\nonumber\\
	&=& \frac{2}N\sum_{j=0}^{M-1} cos\frac{2\pi jk}{M}\sum_{h=0}^{M-1} (y_{2h}cos\frac{2\pi j(2h)}{N} + y_{2h+1}cos\frac{2\pi j(2h+1)}{N}) \nonumber \\
	&=&I_e + II_e, \nonumber
\end{eqnarray}
where
\begin{eqnarray*}
	I_e &=& \frac{2}N\sum_{j=0}^{M-1} cos\frac{2\pi jk}{M}\sum_{h=0}^{M-1} y_{2h}cos\frac{2\pi jh}{M} \nonumber \\
	II_e &=& \frac{2}N\sum_{j=0}^{M-1} cos\frac{2\pi jk}{M}\sum_{h=0}^{M-1} y_{2h+1}cos\frac{2\pi j(2h+1)}{N} 
\end{eqnarray*}
It is not hard to verify 
\[
I_e =y_{2k}, \qquad II_e = \frac{1}M \sum_{h=0}^{M-1} y_{2h+1},
\]
which implies (\ref{eq_keylemma_a_even}). One can similarly derive Eq. (\ref{eq_keylemma_a_odd}).
\begin{eqnarray*}
	\hat{y}_{2k+1} &=& \sum_{j=0}^{M-1} (-1)^j A'_j cos\frac{2\pi j(2k+1)}{N} \nonumber\\
	&=& \sum_{j=0}^{M-1} (-1)^j cos\frac{2\pi j(2k+1)}{N} \frac{2}N\sum_{l=0}^{N-1} (-1)^j y_{l}cos\frac{2\pi jl}{N}\nonumber\\
	&=& \frac{2}N\sum_{j=0}^{M-1} cos\frac{2\pi j(2k+1)}{N}\sum_{h=0}^{M-1} (y_{2h}cos\frac{2\pi j(2h)}{N} + y_{2h+1}cos\frac{2\pi j(2h+1)}{N}) \nonumber \\
	&=&III_e + IV_e, \nonumber
\end{eqnarray*}
where
\begin{eqnarray*}
	III_e &=& \frac{2}N\sum_{j=0}^{M-1} cos\frac{2\pi j(2k+1)}{N}\sum_{h=0}^{M-1} y_{2h}cos\frac{2\pi j (2h)}{N} \nonumber \\
	IV_e &=& \frac{2}N\sum_{j=0}^{M-1} cos\frac{2\pi j(2k+1)}{N}\sum_{h=0}^{M-1} y_{2h+1}cos\frac{2\pi j(2h+1)}{N} 
\end{eqnarray*}
Similarly, we obtain
\[
III_e =\frac{1}M \sum_{h=0}^{M-1} y_{2h}\label{I_odd},  \qquad IV_e =y_{2k+1},
\]
which implies (\ref{eq_keylemma_a_odd}). 

Similar derivation can be used to prove \ref{eq_keylemma_b} if $f(x)$ is odd.

%%




%\end{article}
\end{document}